\documentclass[11pt,twoside, final]{amsart}
\copyrightinfo{0}{Iranian Mathematical Society}
\pagespan{1}{\pageref*{LastPage}}
\usepackage{etoolbox,lastpage}
\commby{}
\usepackage{amsmath,amsthm,amscd,amsfonts,amssymb,enumerate}
\usepackage{graphicx}		
\usepackage{color}
\usepackage[colorlinks]{hyperref}
 
\usepackage{amsfonts,amssymb,amsmath,amsthm}
\usepackage{url}
\usepackage{enumerate}
\usepackage{graphicx}
\usepackage{lipsum}
\usepackage{mathtext}
\usepackage{ifsym}
\usepackage{ifpdf}
\usepackage{latexsym}
\usepackage{amscd}
\usepackage{stmaryrd}

\newtheorem*{ozn}{Notation}

\newtheorem{theorem}{Theorem}[section]

\newtheorem{lemma}[theorem]{Lemma}

\theoremstyle{definition}
\newtheorem{definition}[theorem]{Definition}

\theoremstyle{remark}

\numberwithin{equation}{section}
 
\begin{document}

 
\title[Iterations of Multifunctions]{Iterations of Multifunctions for Graph Theory$\colon$ Bipartite Graphs and Filters} 
 
\author[A. Gi\.zycki]{Artur Gi\.zycki}
\address[Artur Gi\.zycki]{Faculty of Mathematics and Information Science\\
Warsaw University of Technology\\
ul. Koszykowa 75\\
00-662 Warsaw\\
Poland}
\email{a.gizycki@mini.pw.edu.pl}


%
 
 \maketitle
%

\begin{abstract}
We present the theory of multifunctions applied to graphs. Its interesting feature is that walks are recognized as iterations. We consider the graphs with arbitrary number of vertices which are determined by multifunctions. The mutually unique correspondence between graphs and multifunctions is proven. We explain that many facts of graph theory can be formulated in the language of multifunctions and as examples we give$\colon$ neighborhood, walk, independent set, clique, bipartiteness, connectedness, isolated vertices, graph metric, leaf. To simplify the proofs of our theorems, we introduce the concept of iterations of multifunctions. The new equivalent conditions for bipartite multifunctions including the K\"onig theorem and even iterations theorem are given. We prove that there exist filters and ideals in graph theory that are similar to those from the set theory. Finally, to illustrate these facts, we consider the multifunction of prime numbers.
\end{abstract}
 
\section{\bf Introduction}

Determining whether there exists a Hamilton cycle in a graph is a well known $NP-$complete problem. In the article\cite{hussain} there is a new equivalent condition to Hamiltonian graphs. This is expressed in the language of soft sets\cite{softy,hussain} but soft sets are multifunctions. Our article is about the expression in the language of multifunctions if there exists an odd cycle in a graph; or equivalently, by the K\"onig theorem, bipartiteness of a graph. When we consider multifunctions on the same set it would be beautiful to iterate as in dynamical systems. But every such multifunction represents a graph, so iterations too. Making these iterations on the computer caused us to see that if the graph has an odd cycle then its even iterations are disconnected. This observation led us to begin research on iterations of multifunctions.   

The first section is dedicated to determine the mathematical symbols used in this article, define multifunctions, their images and preimages a little differently than in other articles about multifunctions and to listen of basic properties of multifunctions. In Lemmas 2.3 and 2.4 we prove a large number of these properties. We will use them in this article.  

In this article we prove the relationship between graphs and multifunctions in the same way as in the articles\cite{softy,hussain}. We do not call them soft sets but multifunctions. In fact, multifunctions are more natural to represent graphs because many concepts of graph theory are expressed by means of images, preimages, iterations etc. But the iterations are nowhere used in the theory of soft sets perhaps because it is more natural to consider the dynamical system for multifunction.

We prove that walks in graphs are iterations. We define the independent sets, cliques and bipartite graphs. In Lemmas 3.24, 4.12 and Theorems 5.5, 5.8 we prove the new equivalent conditions of bipartiteness. In section 'Iterations of Multifunctions' we prove the basic iteration calculus. The main method of proof is mathematical induction. The results from this section allow us to apply iterations in combinatorial proofs. We define the concept of connectedness. Using simple observation from the theory of numbers we prove the theorem about even iterations.            

We prove that there are filters and ideals in graph theory. We define a new metric on multifunctions that is taken from graph theory and prove the simple fact about Cauchy filters. Using there concepts we discuss about the multifunction of prime numbers.

\section{Basic definitions and facts}
We will use denotations from different books of logic and set theory i.a. \cite{bourbaki,comfort,ullman}. As the case may be, under the symbol of $N$ we mean both $N\cup\{0\}$ and $N-\{0\}$. We will also consider the alphabets that can have infinitely many letters. We fix the following signs.
\begin{ozn}
Let $X\neq\emptyset,n\in N$ and $R\in P(X\times X)-\{\emptyset\}$ be a binary relation. Then$\colon$

\begin{itemize}
\item$P(X)=\{A\mid A\subset X\}$ is powerset of $X$
\item$\uparrow^{c}\colon P(X)\to P(X);A\mapsto\uparrow^{c}(A)=A^{c}=X-A$ is complement
\item$P_{n}(X)=\{A\in P(X)\mid\sharp A=n\}\subset P(X)$ is powerset of $X$ whose elements are $n-$element sets
\item$P_{fin}(X)=\{A\in P(X)\mid\sharp A<\aleph_{0}\}$ is set of all finite subsets of $X$
\item$X^{\star}$ is set of all finite words over alphabet $X$
\item$\text{concat}\colon X^{\star}\times X^{\star}\to X^{\star};(\alpha,\beta)\mapsto\text{concat}(\alpha,\beta)=\alpha\beta$ is concatenation of words
\item$\text{length}\colon X^{\star}\to N;\alpha\mapsto\text{length}(\alpha)=n$ iff the word $\alpha$ has $n$ letters is length function
\item$\downarrow_{n}\colon X^{\star}\to X\cup\{\epsilon\}$ is the letter of $\alpha$ with number $n$ i.e.\newline$\alpha\mapsto\alpha_{n}=\begin{cases}x$ iff $\text{length}(\alpha)\ge n$ and the $n$-th letter of $\alpha$ is $x\\\epsilon$ iff $\text{length}(\alpha)<n\end{cases}$
\item$X^{\star}_{n}=\{\alpha\in X^{\star}\mid\text{length}(\alpha)=n\}\subset X^{\star}$ is set of words with $n$ letters
\item$\Delta_{X\times X}=\{(x_{1},x_{2})\in X\times X\mid x_{1}=x_{2}\}\subset X\times X$ is diagonal of set $X$
\item$R^{-1}=\{(x_{1},x_{2})\in X\times X\mid(x_{2},x_{1})\in R\}\subset X\times X$ is inverse of $R$
\item$\tilde{R}=\{\{x_{1},x_{2}\}\in P(X)\mid(x_{1},x_{2})\in R\cup R^{-1}\}\subset P(X)$ is disorder of $R$
\item$\alpha(X)=\{\Phi\in P(P(X))-\{\emptyset\}\mid\forall A,B\in P(X)\colon A\in\Phi\land B\in F\Leftrightarrow A\cap B\in\Phi\}$ is set of filters of $X$
\end{itemize}

\end{ozn}
The simplest fact is that $\uparrow^{c}\circ\uparrow^{c}=\text{id}_{P(X)}$. We have the following easy facts about alphabets$\colon\hspace{1pt}X^{\star}_{0}=\{\epsilon\},\hspace{1pt}X^{\star}_{1}\simeq X,\hspace{1pt}\bigcup_{n\in N}X^{\star}_{n}=X^{\star}$ and $(X^{\star}_{n})_{n\in N}$ is pairwise disjoint. 

It is easily seen that $\tilde{R}\subset P_{1}(X)\dot{\cup}P_{2}(X)$. Notice that if the relation $R$ is symmetric, i.e. $R=R^{-1}$, then $\tilde{R}=\{\{x,y\}\in P_{1}(X)\dot{\cup}P_{2}(X)\mid(x,y)\in R\}$.

Recall the basic definitions of multifunction theory. As a multifunction $F$ from $X\neq\emptyset$ to $Y\neq\emptyset$ we understand the function $F\colon X\to P(Y);x\mapsto F(x)\subset Y$, that for simplicity we denote $F\colon X\leadsto Y$. The set of all multifunctions $F\colon X\leadsto Y$ is equal $P(Y)^{X}$. We do not consider multifunctions with $X=\emptyset$ or $Y=\emptyset$, even if we think about trivial multifunction, and therefore we always assume that $X\neq\emptyset$ and $Y\neq\emptyset$.
\begin{definition}\cite{aubin,frych}
Let $F,G\colon X\leadsto Y$ be multifunctions, $S\subset Y$ and $\boxcircle\colon P(Y)\to P(Y);(S_{1},S_{2})\mapsto S_{1}\boxcircle S_{2}$ be a set operation. Then$\colon$ 
\begin{itemize}
\item$F$ is nontrivial iff $\exists x\in X\colon F(x)\neq\emptyset$
\item$F$ is strict iff $\forall x\in X\colon F(x)\neq\emptyset$
\item$F^{c}\colon X\to P(Y);x\mapsto F^{c}(x)=Y-F(x)$ is completion of $F$
\item$D(F)=\{x\in X\mid F(x)\neq\emptyset\}\subset X$ is a domain of $F$
\item$\reflectbox{D}(F)=\bigcup_{x\in X}F(x)$ is a codomain of $F$
\item$F^{-1}\colon Y\to P(X);y\mapsto F^{-1}(y)=\{x\in X\mid y\in F(x)\}$ is an inverse of $F$
\item$\{\cdotp\}\colon X\to P(X);x\mapsto\{x\}$ is a singleton multifunction
\item$\text{const}^{S}\colon X\to P(Y);x\mapsto\text{const}^{S}(x)=S$ is a constant multifunction
\item$F\subset G$ iff $\forall x\in X\colon F(x)\subset G(x)$
\item$F\boxcircle G\colon X\to P(Y);x\mapsto F\boxcircle G(x)=F(x)\boxcircle G(x)$ is a multifunction operation 
\end{itemize}

\end{definition}
It is easily seen that for every multifunctions $F,G,H\colon X\leadsto Y$ the following equalities are fulfilled$\colon\hspace{1pt}(F\cup G)^{-1}=F^{-1}\cup G^{-1},\hspace{1pt}(F\cap G)^{-1}=F^{-1}\cap G^{-1},\hspace{1pt}(F^{-1})^{-1}=F=F^{cc},\hspace{1pt}F\cup G=G\cup F,\hspace{1pt}F\cap G=G\cap F,\hspace{1pt}F\cap(G\cup H)=(F\cap G)\cup(F\cap H),\hspace{1pt}(F^{c})^{-1}=(F^{-1})^{c},\hspace{1pt}(\text{const}^{\emptyset})^{-1}=\text{const}^{\emptyset},\hspace{1pt}(\text{const}^{Y})^{-1}=\text{const}^{Y},\hspace{1pt}(\text{const}^{\emptyset})^{c}=\text{const}^{Y},\hspace{1pt}(\text{const}^{Y})^{c}=\text{const}^{\emptyset},F\subset G\Leftrightarrow G^{c}\subset F^{c}$ and if $F$ is strict then $D(F)$ is the domain of function $F\in P(Y)^{X}$. Almost the same we can prove for the sequence of multifunction $(F_{n})\colon N\to P(Y)^{X};n\mapsto F_{n}\colon X\to P(Y);x\mapsto F_{n}(x)$ e.g. $(\bigcup_{n\in N}F_{n})^{-1}=\bigcup_{n\in N}F^{-1}_{n}$.

For multifunctions as well as for functions we can define images and preimages. For the purpose of this article, in order to not to get lost in the signs, we change the standard signs from $F^{-1},F^{+1},F$ to $F_{-},F_{+},F_{\cup}$.
\begin{definition}\cite{aubin,bergetop,frych}
Let $F\colon X\leadsto Y$ be a multifunction. Then$\colon$
\begin{itemize}
\item$F_{-}\colon P(Y)\to P(X);B\mapsto F_{-}(B)=\{x\in X\mid F(x)\cap B\neq\emptyset\}\subset X$ is a complete preimage of $F$
\item$F_{+}\colon P(Y)\to P(X);B\mapsto F_{+}(B)=\{x\in X\mid F(x)\subset B\}\subset X$ is a small preimage of $F$
\item$F_{\cup}\colon P(X)\to P(Y);A\mapsto F_{\cup}(A)=\{y\in Y\mid\exists a\in A\colon y\in F(a)\}=\bigcup_{a\in A}F(a)\subset Y$ is an $\cup-$image of $F$
\item$F_{\cap}\colon P(X)-\{\emptyset\}\to P(Y);A\mapsto F_{\cap}(A)=\{y\in Y\mid\forall a\in A\colon y\in F(a)\}=\bigcap_{a\in A}F(a)\subset Y$ is an $\cap-$image of $F$

\end{itemize}

\end{definition}
It is obvious that $F_{\cup}(X)=\reflectbox{D}(F),F\subset G\Rightarrow F_{-}\subset G_{-},F\subset G\Rightarrow F_{\cup}\subset G_{\cup}$ and $F\subset G\Rightarrow F_{+}\supset G_{+}$. With the assumption that a multifunction $F$ is strict we have that $F_{+}(B)\subset F_{-}(B)$\cite{frych}. The assumption of strictness is important in this article. To go through the logic we will use the immediate fact $A\cap B=\emptyset$ iff $A\subset B^{c}$.

Recall the interesting properties of images and preimages of multifunctions with immediate proofs. 
\begin{lemma}\cite{aubin,bergetop,frych}
Let $F\colon X\leadsto Y$ be a multifunction and $A\subset X,B\subset Y$. Then$\colon$

\begin{enumerate}[(i)]
\item$F_{+}(Y)=X$
\item$F_{-}(\emptyset)=\emptyset$
\item$F^{-1}_{\cup}=F_{-}=\uparrow^{c}\circ F_{+}\circ\uparrow^{c}$
\item$\text{const}^{\emptyset}_{-}(B)=\emptyset$
\item$\text{const}^{\emptyset}_{+}(B)=X$
\item$\text{const}^{Y}_{-}(B)=\begin{cases}X$ iff $B\neq\emptyset\\\emptyset$ iff $B=\emptyset\end{cases}$
\item$\text{const}^{Y}_{+}(B)=\begin{cases}X$ iff $B=Y\\\emptyset$ iff $B\neq Y\end{cases}$
\item$F$ is trivial iff $F_{-}(Y)=\emptyset$
\item$F$ is strict iff $F_{+}(\emptyset)=\emptyset$
\item$F$ is trivial iff $F_{+}(\emptyset)=X$
\item$F$ is strict iff $F_{-}(Y)=X$
\item$F$ is not strict iff $\exists C\in P(X)-\{\emptyset\}\colon F|_{C}=\text{const}^{\emptyset}$
\item$F_{+}(B)=X$ iff $B\supset \reflectbox{D}(F)$
\item$\exists(x,y)\in A\times B\colon y\in F(x)$ iff $A\cap F_{-}(B)\neq\emptyset$
\item$F_{\cup}(A)\neq\emptyset$ iff $A\cap F_{-}(Y)\neq\emptyset$
\item$F_{\cup}(\emptyset)=\emptyset$
\item$A=\emptyset$ iff $F_{\cup}(A)=\emptyset$ assuming that $F$ is nontrivial
\item$F_{-}(B)=\emptyset$ iff $\forall(x,y)\in X\times B\colon y\notin F(x)$
\item$F_{+}(B)=\emptyset$ iff $\forall x\in X\colon F(x)\cap B^{c}\neq\emptyset$
\item$F_{-}(B^{c})=F_{+}(B)^{c}$
\item$F_{+}(B^{c})=F_{-}(B)^{c}$
\item$\{\cdot\}_{\cup}(A)=\{\cdot\}_{-}(A)=\{\cdot\}_{+}(A)=A$
\item$\{\cdot\}_{\cap}(A)=\begin{cases}A$ iff $\sharp(A)=1\\\emptyset$ iff $\sharp A>1\end{cases}$
\item$F_{\cup}\circ\{\cdot\}=F=F_{\cap}\circ\{\cdot\}$
\item$F_{-}\circ\{\cdot\}=F^{-1}$
\item$F_{\cup}=\uparrow^{c}\circ F^{c}_{\cap}$
\item$F_{\cap}=\uparrow^{c}\circ F^{c}_{\cup}$
\item$F^{-1}_{\cap}=\uparrow^{c}\circ F^{c}_{-}$
\end{enumerate}

\end{lemma}

\begin{proof}
Fix $x\in X,y\in Y,A\subset X,B\subset Y$.
\begin{enumerate}[(i)]
\item$F_{+}(Y)=\{x\in X\mid F(x)\subset Y\}=X$.
\item$F_{-}(\emptyset)=\{x\in X\mid F(x)\cap\emptyset=\emptyset\neq\emptyset\}=\emptyset$.
\item$F_{-}(B)=\{x\in X\mid B\cap F(x)\neq\emptyset\}=\{x\in X\mid\exists y\in Y\colon y\in B\cap F(x)\}=\{x\in X\mid\exists y\in B\colon y\in F(x)\}=\{x\in X\mid\exists y\in B\colon x\in F^{-1}(y)\}=\bigcup_{y\in B}F^{-1}(y)=F^{-1}_{\cup}(B)$ and $(\uparrow^{c}\circ F_{+}\circ\uparrow^{c})(B)=(F_{+}(B^{c}))^{c}=\{x\in X\mid F(x)\subset B^{c}\}^{c}=\{x\in X\mid F(x)\not\subset B^{c}\}=\{x\in X\mid F(x)\cap B\neq\emptyset\}=F_{-}(B)$.
\item$\text{const}^{\emptyset}_{-}(B)=\{x\in X\mid\text{const}^{\emptyset}(x)\cap B\neq\emptyset\}=\{x\in X\mid\emptyset\cap B=\emptyset\neq\emptyset\}=\emptyset$.
\item$\text{const}^{\emptyset}_{+}(B)=\{x\in X\mid\text{const}^{\emptyset}(x)\subset B\}=\{x\in X\mid\emptyset\subset B\}=X$.
\item$\text{const}^{Y}_{-}(B)=\{x\in X\mid\text{const}^{Y}(x)\cap B\neq\emptyset\}=\{x\in X\mid Y\cap B\neq\emptyset\}=$\newline$=\{x\in X\mid B\neq\emptyset\}=\begin{cases}X$ iff $B\neq\emptyset\\\emptyset$ iff $B=\emptyset\end{cases}$
\item$\text{const}^{Y}_{+}(B)=\{x\in X\mid\text{const}^{Y}(x)\subset B\}=\{x\in X\mid B=Y\}=\begin{cases}X$ iff $B=Y\\\emptyset$ iff $B\neq Y\end{cases}$
\item$F$ is nontrivial iff $\exists x\in X\colon F(x)\cap Y=F(x)\neq\emptyset$. We may say equivalently that $\exists x\in X\colon x\in F_{-}(Y)$ i.e. $F_{-}(Y)\neq\emptyset$.
\item Notice that $F_{+}(\emptyset)=\{x\in X\mid F(x)=\emptyset\}$. But $F$ is nonstrict iff $\exists x\in X\colon F(x)=\emptyset$; or equivalently, $\exists x\in X\colon x\in F_{+}(\emptyset)$ i.e. $F_{+}(\emptyset)\neq\emptyset$.
\item$F$ is trivial iff $\emptyset=F_{-}(Y)=(\uparrow^{c}\circ F_{+}\circ\uparrow^{c})(Y)=F_{+}(Y^{c})^{c}=F_{+}(\emptyset)^{c}$; or equivalently, $X=\emptyset^{c}=F_{+}(\emptyset)^{cc}=F_{+}(\emptyset)$.
\item$F$ is strict iff $\emptyset=F_{+}(\emptyset)=(\uparrow^{c}\circ F_{-}\circ\uparrow^{c})(\emptyset)=F_{-}(\emptyset^{c})^{c}=F_{-}(Y)^{c}$; or equivalently, $X=\emptyset^{c}=F_{-}(Y)^{cc}=F_{-}(Y)$.
\item$F$ is not strict iff $F_{+}(\emptyset)\neq\emptyset$. We may say equivalently that $\exists C\in P(X)-\{\emptyset\}\colon C\subset F_{+}(\emptyset)$; or equivalently, $\exists C\in P(X)-\{\emptyset\}\colon\forall x\in C\colon F(x)=\emptyset$ i.e. $\exists C\in P(X)-\{\emptyset\}\colon F|_{C}=\text{const}^{\emptyset}$.
\item$F_{+}(B)=X$ iff $\forall x\in X\colon F(x)\subset B$; or equivalently, $B\supset\bigcup_{x\in X}F(x)=F_{\cup}(X)=\reflectbox{D}(F)$.
\item$A\cap F_{-}(B)\neq\emptyset$ iff there exists $x\in X$ such that $x\in A$ and $F(x)\cap B\neq\emptyset$. We may say equivalently that there exist $y\in Y$ and $x\in X$ such that $x\in A$ and $y\in B$ and $y\in F(x)$ i.e. $\exists(x,y)\in A\times B\colon y\in F(x)$.
\item$F_{\cup}(A)\neq\emptyset$ iff there exists $y\in Y$ such that $y\in F_{\cup}(A)=\bigcup_{x\in A}F(x)$. We may say equivalently that there exists $(x,y)\in A\times Y$ such that $y\in F(x)$ i.e. $A\cap F_{-}(Y)\neq\emptyset$
\item$F_{\cup}(\emptyset)=\bigcup_{a\in\emptyset}F(a)=\emptyset$.
\item$F_{\cup}(A)=\emptyset$ iff $A\cap F_{-}(Y)=\emptyset$. If we assume that $F$ is nontrivial then $A\cap F_{-}(Y)=\emptyset$ iff $A=\emptyset$.
\item Notice that $F_{-}(B)=X\cap F_{-}(B)\neq\emptyset$ iff $\exists(x,y)\in X\times B\colon y\in F(x)$. We may say equivalently that $F_{-}(B)=\emptyset$ iff $\forall(x,y)\in X\times B\colon y\notin F(x)$.
\item$F_{+}(B)=\emptyset$ iff $\forall x\in X\colon x\notin F_{+}(B)$; or equivalently, $\forall x\in X\colon F(x)\not\subset B$ i.e. $\forall x\in X\colon F(x)\cap B^{c}\neq\emptyset$.
\item$F_{-}(Y-B)=F_{-}(B^{c})=(F_{-}\circ\uparrow^{c})(B)=(\uparrow^{c}\circ F_{+})(B)=F_{+}(B)^{c}=X-F_{+}(B)$.
\item$F_{+}(Y-B)=X-(X-F_{+}(Y-B))=X-F_{-}(Y-(Y-B))=X-F_{-}(B)$.
\item$\{\cdot\}_{-}(A)=\{x\in X\mid\{x\}\cap A\neq\emptyset\}=\{x\in X\mid x\in A\}=A=\{x\in X\mid\{x\}\subset A\}=\{\cdot\}_{+}(A)$ and $A=\bigcup_{x\in A}\{x\}=\{\cdot\}_{\cup}(A)$ because $\{x\}\subset A$ iff $x\in A$ and $x\in A$ iff $\{x\}\cap A\neq\emptyset$.
\item$\{\cdot\}(\{x\})=\bigcap_{y\in\{x\}}\{y\}=\{x\}$ and $\{\cdot\}(A)=\bigcap_{x\in A}\{x\}=\emptyset$ iff $\sharp A>1$.
\item$(F_{\cup}\circ\{\cdot\})(x)=F_{\cup}(\{x\})=\bigcup_{y\in\{x\}}F(y)=F(x)=\bigcap_{y\in\{x\}}F(y)=F_{\cap}(\{x\})=(F_{\cap}\circ\{\cdot\})(x)$.
\item$(F_{-}\circ\{\cdot\})(y)=F_{-}(\{y\})=\{x\in X\mid\{y\}\cap F(x)\neq\emptyset\}=\{x\in X\mid y\in F(x)\}=\{x\in X\mid x\in F^{-1}(y)\}=F^{-1}(y)$.
\item$F_{\cup}(A)=\bigcup_{a\in A}F(a)=\bigcup_{a\in A}F^{cc}(a)=\bigcup_{a\in A}(F^{c}(a))^{c}=\bigcup_{a\in A}(Y-F^{c}(a))=Y-\bigcap_{a\in A}(F^{c}(a))=Y-F^{c}_{\cap}(A)=F^{c}_{\cap}(A)^{c}=(\uparrow^{c}\circ F^{c}_{\cap})(A)$.
\item$F_{\cap}=F^{cc}_{\cap}=((F^{c})^{c})_{\cap}=\uparrow^{c}\circ\uparrow^{c}\circ((F^{c})^{c})_{\cap}=\uparrow^{c}\circ(\uparrow^{c}\circ((F^{c})^{c})_{\cap})=\uparrow^{c}\circ F^{c}_{\cup}$.
\item$F^{-1}_{\cap}=\uparrow^{c}\circ(F^{-1})^{c}_{\cup}=\uparrow^{c}\circ(F^{c})^{-1}_{\cup}=\uparrow^{c}\circ F^{c}_{-}$.
\end{enumerate}
\end{proof}

The strictness of one multifunction goes to the strictness of the union of multifunctions.
\begin{lemma}
Let $(F_{t})_{t\in T}\colon T\to P(Y)^{X}$ be a family of multifunctions and $B\subset Y$. Then$\colon$
\begin{enumerate}[(i)]
\item$\bigcap_{t\in T}(F_{t})_{+}(B)=(\bigcup_{t\in T}F_{t})_{+}(B)$.
\item$\exists t\in T\colon F_{t}$ is strict iff $\bigcup_{t\in T}F_{t}$ is strict.
\item$\forall t\in T\colon F_{t}$ is not strict iff $\bigcup_{t\in T}F_{t}$ is not strict.
\end{enumerate}
\end{lemma}

\begin{proof}
Fix $x\in X$.
\begin{enumerate}[(i)]
\item$x\in\bigcap_{t\in T}(F_{t})_{+}(B)$ iff $\forall t\in T\colon x\in(F_{t})_{+}(B)$; or equivalently, $\forall t\in T\colon F_{t}(x)\subset B$. We may say equivalently that $\bigcup_{t\in T}F_{t}(x)\subset B$ i.e. $x\in(\bigcup_{t\in T}F_{t})_{+}(B)$.
\item$\bigcup_{t\in T}F_{t}$ is strict iff $\emptyset=(\bigcup_{t\in T}F_{t})_{+}(\emptyset)=\bigcap_{t\in T}(F_{t})_{+}(\emptyset)$; or equivalently, $\exists t\in T\colon(F_{t})_{+}(\emptyset)=\emptyset$ i.e. $\exists t\in T\colon F_{t}$ is strict.
\item$\bigcup_{t\in T}F_{t}$ is not strict iff not $\exists t\in T\colon F_{t}$ is strict i.e. $\forall t\in T\colon F_{t}$ is not strict.
\end{enumerate}
\end{proof}

\section{Graphs as multifunctions}

We want to describe a whole graph theory (i.a. determine whether a given simple graph is bipartite) in terms of multifunctions. We believe that it is possible. At the beginning we will have to make some changes to what we were tought about graphs.
\subsection{Types of graphs}
The concept of graphs is understood in the following ways.   
\begin{definition}\cite{bollobas,jung}
Let $E,\hspace{1pt}V$ be sets and $V\neq\emptyset$. Then$\colon$
\begin{itemize}
\item$G=(V,E)$ is a simple graph iff $E\subset P_{2}(V)$
\item$G=(V,E)$ is an undirected graph iff $E\subset P_{1}(V)\dot{\cup}P_{2}(V)$ 
\item$G=(V,E)$ is an everywhereloop undirected graph iff $P_{1}(V)\subset E\subset P_{1}(V)\dot{\cup}P_{2}(V)$ 
\item$G=(V,E)$ is a digraph iff $E\subset V\times V$
\item$G=(V,E)$ is a simple digraph iff $E\subset V\times V$ and $\Delta_{V\times V}\cap E=\emptyset$ 
\item$G=(V,E)$ is an everywhereloop digraph iff $\Delta_{V\times V}\subset E\subset V\times V$
\item$G=(V,E)$ is an orgraph iff $G=(V,E)$ is a simple digraph and $E\cap E^{-1}=\emptyset$
\end{itemize}
\end{definition}
We define the basic types of multifunctions which will be used to join the graph theory and the multifunction theory.
\begin{definition}\cite{aubin,kechris}
Let $V\neq\emptyset$ and $F\colon V\leadsto V$ be a multifunction. Then$\colon$
\begin{itemize}
\item$F$ is undirected iff $F=F^{-1}$
\item$F$ is oriented iff $F\cap F^{-1}=\text{const}^{\emptyset}$
\item$F$ is total iff $F\cup F^{-1}=\text{const}^{V}$
\item$F$ is everywhereloop iff $\{\cdotp\}\subset F$
\item$F$ is loopless iff $\{\cdotp\}\cap F=\text{const}^{\emptyset}$
\item$F$ is a simple graph multifunction iff $F$ is loopless and undirected
\item$F$ is an orgraph multifunction iff $F$ is loopless and oriented
\item$F$ is transitive iff $\forall u,v,w\in V\colon u\in F(v)\land v\in F(w)\Rightarrow u\in F(w)$
\item$E_{F}=\{(u,v)\in V\times V\mid u\in F(v)\}\subset V\times V$ is a picture of the multifunction $F$ or graph of multifunction
\item$\partial_{(F,V)}\colon P(V)\to P(V);W\mapsto\partial_{(F,V)}(W)=W\cap F_{-}(W^{c})$ is a boundary with respect to $F$ 
\item$F$ is amenable iff $\forall\epsilon>0\colon\exists W\in P_{fin}(V)-\{\emptyset\}\colon\frac{\sharp{\partial_{(F,V)}(W)}}{\sharp{W}}<\epsilon$
\end{itemize}

\end{definition}
It is easy to prove that for every set $V\neq\emptyset$ and for every multifunctions $F,G\colon V\leadsto V$ the following facts are fulfilled$\colon\hspace{1pt}F\cap\text{const}^{V}=F$, if $F$ is oriented then $F\cap G$ is oriented, if $F$ is loopless then $F\cap G$ is loopless, if $F$ is orgraph multifunction then $F\cap G$ is orgraph multifunction, $E_{F}^{-1}=E_{F^{-1}}$, $F$ is undirected iff $E_{F}$ is symmetric, $F$ is loopless iff $E_{F}$ is irreflexive, $F$ is everywhereloop iff $E_{F}$ is reflexive, $F$ is oriented iff $E_{F}$ is asymmetric, $F$ is transitive iff $E_{F}$ is transitive, $F-\{\cdotp\}$ is loopless, $F\cup\{\cdotp\}$ is everywhereloop, if $F$ is total then $F$ is nontrivial.

There is a mutually unique correspondence between graphs and multifunctions. The essence of this link is the neighborhood. 
\subsection{Correspondence, neighborhood and selection}
It is known that all graphs are either directed or undirected. For both we consider the neighborhood.
\begin{definition}\cite{diestel,jung}
Let $G=(V,E)$ be an undirected graph. Then $N_{G}\colon V\to P(V);v\mapsto N_{G}(v)=\{u\in V\mid\{u,v\}\in E\}$ is the neighborhood multifunction.
\end{definition}

\begin{definition}\cite{diestel,jung}
Let $G=(V,E)$ be a digraph. Then $N_{G}^{-}\colon V\to P(V);v\mapsto N_{G}^{+}(v)=\{u\in V\mid(u,v)\in E\}$ is the inneighborhood multifunction and $N_{G}^{+}\colon V\to P(V);v\mapsto N_{G}^{-}(v)=\{u\in V\mid(v,u)\in E\}$ is the outneighborhood multifunction.
\end{definition}

For undirected graphs we have the following lemmas.
\begin{lemma}\cite{softy,hussain}
Let $G=(V,E)$ be an undirected graph. Then$\colon$
\begin{enumerate}[(i)]
\item$N_{G}$ is undirected multifunction 
\item if $G$ is simple graph then $N_{G}$ is simple graph multifunction
\item if $G$ is everywhereloop then $N_{G}$ is everywhereloop
\end{enumerate}
\end{lemma}

\begin{proof}
Fix $u,v\in V$.
\begin{enumerate}[(i)]
\item$u\in N_{G}^{-1}(v)$ iff $\{v,u\}\in E$; or equivalently, $\{u,v\}\in E$ and so $v\in N_{G}(u)$ i.e. $N_{G}$ is undirected.
\item Assume that $E\subset P_{2}(V)$ and $N_{G}$ is not loopless. Then there exists $v\in V$ such that $v\in N_{G}(v)$; or equivalently, there exists $v\in V$ such that $\{v,v\}\in E\subset P_{2}(V)$. But $\{v,v\}=\{v\}\notin P_{2}(V)$, contradiction. $N_{G}$ is simple graph multifunction because $N_{G}$ is undirected.
\item Assume that $P_{1}(V)\subset E\subset P_{1}(V)\dot{\cup}P_{2}(V)$. Then $\{v\}=\{v,v\}\in E$; or equivalently, $v\in N_{G}(v)$ and so $\{v\}\subset N_{G}(v)$ i.e. $\{\cdotp\}\subset N_{G}$.
\end{enumerate}

\end{proof}

\begin{lemma}\cite{softy,hussain}
Let $V\neq\emptyset$ and $F\colon V\leadsto V$ be a multifunction. Then$\colon$
\begin{enumerate}[(i)]
\item if $F$ is undirected then $\tilde{G_{F}}=(V,\tilde{E_{F}})$ is undirected graph with set of edges $\tilde{E_{F}}=\{\{u,v\}\in P_{1}(V)\dot{\cup}P_{2}(V)\mid u\in F(v)\}$
\item if $F$ is simple graph multifunction then $\tilde{G_{F}}=(V,\tilde{E_{F}})$ is simple graph with set of edges $\tilde{E_{F}}=\{\{u,v\}\in P_{2}(V)\mid u\in F(v)\}$
\item if $F$ is everywhereloop and undirected then $\tilde{G_{F}}=(V,\tilde{E_{F}})$ is everywhereloop undirected graph with the set of edges\newline$\tilde{E_{F}}=\{\{u,v\}\in P_{1}(V)\dot{\cup}P_{2}(V)\mid u\in F(v)\}$
\end{enumerate}

\end{lemma}

\begin{proof}
Fix $v\in V$.
\begin{enumerate}[(i)]
\item$F$ is undirected iff $E_{F}$ is symmetric. Then\newline$\tilde{E_{F}}=\{\{u,v\}\in P_{1}(V)\dot{\cup}P_{2}(V)\mid(u,v)\in E_{F}\}=$\newline$=\{\{u,v\}\in P_{1}(V)\dot{\cup}P_{2}(V)\mid u\in F(v)\}\subset P_{1}(V)\dot{\cup}P_{2}(V)$ i.e. $\tilde{G_{F}}=(V,\tilde{E_{F}})$ is undirected graph.
\item$F$ is simple graph multifunction iff $E_{F}$ is irreflexive and symmetric. Then $\tilde{E_{F}}\subset P_{1}(V)\dot{\cup}P_{2}(V)$. But $P_{1}(V)\cap\tilde{E_{F}}=\emptyset$, because $\{v\}=\{v,v\}\notin\tilde{E_{F}}$. Therefore $\tilde{E_{F}}\subset P_{2}(V)$ i.e. $\tilde{G_{F}}=(V,\tilde{E_{F}})$ is simple graph.
\item$F$ is undirected and everywhereloop iff $E_{F}$ is symmetric and reflexive. Then $\tilde{E_{F}}\subset P_{1}(V)\dot{\cup}P_{2}(V)$. But $P_{1}(V)\subset\tilde{E_{F}}$, because $\{v\}=\{v,v\}\in\tilde{E_{F}}$. Therefore $\tilde{G_{F}}=(V,\tilde{E_{F}})$ is everywhereloop undirected graph.
\end{enumerate}
\end{proof}

For directed graphs we have the following lemmas.
\begin{lemma}\cite{softy,hussain}
Let $G=(V,E)$ be a digraph. Then$\colon$
\begin{enumerate}[(i)]
\item if $G$ is simple digraph then $N_{G}^{+}$ and $N_{G}^{-}$ are loopless multifunctions
\item if $G$ is everywhereloop digraph then $N_{G}^{+}$ and $N_{G}^{-}$ are everywhereloop multifunctions
\item if $G$ is orgraph then $N_{G}^{+}$ and $N_{G}^{-}$ are orgraph multifunctions

\end{enumerate}

\end{lemma}

\begin{proof}
Fix $v\in V$.
\begin{enumerate}[(i)]
\item Assume that $\Delta_{V\times V}\cap E=\emptyset$ and $E\subset V\times V$ and exists $v\in V$ such that $v\in N_{G}^{+}(v)$. Then exists $v\in V$ such that $(v,v)\in E\cap\Delta_{V\times V}=\emptyset$, contradiction.
\item Asssume that $\Delta_{V\times V}\subset E\subset V\times V$. Then $(v,v)\in\Delta_{V\times V}\subset E$ and so $v\in N_{G}^{+}(v)$; or equivalently, $\{\cdotp\}\subset N_{G}^{+}$ i.e. $N_{G}^{+}$ is everywhereloop.
\item Assume that $\Delta_{V\times V}\cap E=\emptyset,\hspace{1pt}E\subset V\times V$ and $E\cap E^{-1}=\emptyset$. Then $N_{G}^{+}$ is loopless. Assume that $N_{G}^{+}\cap(N_{G}^{+})^{-1}\neq\text{const}^{\emptyset}$. Then there exist $u,v\in V$ such that $u\in N_{G}^{+}(v)\cap(N_{G}^{+})^{-1}(v)$; or equivalently, there exist $u,v\in V$ such that $u\in N_{G}^{+}(v)$ and $v\in N_{G}^{+}(u)$. We may say equivalently that there exist $u,v\in V$ such that $(u,v)\in E$ and $(v,u)\in E$ i.e. $(u,v)\in E\cap E^{-1}$ and so $E\cap E^{-1}\neq\emptyset$, contradiction. Therefore $N_{G}^{+}$ is loopless and oriented, or orgraph multifunction.
\end{enumerate}
The same arguments pass to $N_{G}^{-}$.
\end{proof}

\begin{lemma}\cite{softy,hussain}
Let $V\neq\emptyset$ and $F\colon V\leadsto V$ be a multifunction. Then$\colon$
\begin{enumerate}[(i)]
\item$G_{F}=(V,E_{F})$ is digraph
\item if $F$ is loopless then $G_{F}=(V,E_{F})$ is simple digraph
\item if $F$ is everywhereloop then $G_{F}=(V,E_{F})$ is everywhereloop digraph
\item if $F$ is orgraph multifunction then $G_{F}=(V,E_{F})$ is orgraph
\end{enumerate}

\end{lemma}
\begin{proof}

\begin{enumerate}[(i)]
\item$E_{F}=\{(u,v)\in V\times V\mid u\in F(v)\}\subset V\times V$ and so $G_{F}=(V,E_{F})$ is digraph.
\item$F$ is loopless iff $E_{F}$ is irreflexive i.e. $\Delta_{V\times V}\cap E_{F}=\emptyset$ and so $G_{F}=(V,E_{F})$ is simple digraph.
\item$F$ is everywhereloop iff $E_{F}$ is reflexive i.e. $\Delta_{V\times V}\subset E_{F}$ and so $G_{F}=(V,E_{F})$ is everywhereloop digraph. 
\item$F$ is orgraph multifunction iff $E_{F}$ is irreflexive and asymmetric i.e. $\Delta_{V\times V}\cap E_{F}=\emptyset$ and $E_{F}\cap E_{F}^{-1}=\emptyset$ and so $G_{F}=(V,E_{F})$ is orgraph.
\end{enumerate}
\end{proof}

For this reason, we consider the graphs $G=(V,E)$ as multifunctions $F\colon V\leadsto V$ i.e. $G_{F}=(V,E_{F})$ or $\tilde{G_{F}}=(V,\tilde{E_{F}})$. From this point, we assume that everywhere there is the additional assumption $V\neq\emptyset$. We do not consider the graph with empty set of vertices.

If we have the graph drawn on a paper we can easily make it oriented and undirected by painting the arrows. We can do this symbolically for multifunctions. If we have an oriented multifunction we can make the undirected multifunction and conversely. For this we will need the concept of selection derived from biological science.
\begin{definition}\cite{king}
Let $V$ be a set. Then $Select(V)=\{\phi\in V^{P_{2}(V)}\mid\forall A\in P_{2}(V)\colon\phi(A)\in A\}\subset V^{P_{2}(V)}$ is the set of selections on $V$.

\end{definition}
Having the selection we can make the multifunction.
\begin{definition}
Let $V$ be a set and $\phi\in Select(V)$. Then $\hat{\phi}\colon V\to P(V);v\mapsto\hat{\phi}(v)=\{u\in V\mid\phi(\{u,v\})=v\}$ is the multifunction called multiselection of $\phi$.
\end{definition}
Notice that for every set $V$ and every selection $\phi\in Select(V)$ the multiselection $\hat{\phi}\colon V\leadsto V$ is nontrivial. 
\begin{lemma}
Let $F\colon V\leadsto V$ be a multifunction and $\phi\in Select(V)$. Then$\colon$
\begin{enumerate}[(i)]
\item$\hat{\phi}$ is nontrivial total orgraph multifunction
\item$F\cup F^{-1}$ is undirected
\item$F\cap\hat{\phi}$ is orgraph multifunction
\item if $F$ is undirected then $F=(F\cap\hat{\phi})\dot{\cup}(F\cap\hat{\phi})^{-1}$
\end{enumerate}

\end{lemma}

\begin{proof}
Fix $u,v\in V$. 
\begin{enumerate}[(i)]
\item Notice that $u\in\hat{\phi}(v)\cup\hat{\phi}^{-1}(v)$ iff $u\in\hat{\phi}(v)$ or $v\in\hat{\phi}(u)$. We may say equivalently that $u=\phi(\{u,v\})$ or $\phi(\{u,v\})=v$. But for every $u,v\in V$ the following formula holds $\phi(\{u,v\})\in\{u,v\}$. Therefore $\hat{\phi}\cup\hat{\phi}^{-1}=\text{const}^{V}$ i.e. $\hat{\phi}$ is total. Assume that there exists $v\in V$ such that $v\in\hat{\phi}(v)$. Then $\phi(\{v,v\})=v$ but $\phi(\{v,v\})\notin V$ because $\{v,v\}=\{v\}\in P_{1}(V)$ and domain of $\phi$ is $P_{2}(V)$, contradiction. Therefore $\hat{\phi}$ is loopless. Assume that $\hat{\phi}$ is not oriented i.e. $\hat{\phi}\cap\hat{\phi}^{-1}\neq\text{const}^{\emptyset}$. Then there exist $u,v\in V$ such that $u\in\hat{\phi}(v)\cap\hat{\phi}^{-1}(v)$; or equivalently, $u\in\hat{\phi}(v)$ and $v\in\hat{\phi}(u)$ i.e. $u=\phi(\{u,v\})=v$. Thus $v\in\hat{\phi}(v)\cap\hat{\phi}^{-1}(v)\subset\hat{\phi}(v)$ and so $v\in\hat{\phi}(v)$ i.e. $\hat{\phi}$ is not loopless, contradiction. Therefore $\hat{\phi}$ is oriented but every total multifunction is nontrivial.
\item$(F\cup F^{-1})^{-1}=F^{-1}\cup(F^{-1})^{-1}=F^{-1}\cup F=F\cup F^{-1}$ and so $F\cup F^{-1}$ is undirected. 
\item$\hat{\phi}$ is orgraph multifunction and so $F\cap\hat{\phi}$ is orgraph multifunction.
\item If $F=F^{-1}$ then $(F\cap\hat{\phi})\cup(F\cap\hat{\phi})^{-1}=(F\cap\hat{\phi})\cup(F^{-1}\cap\hat{\phi}^{-1})=
(F\cap\hat{\phi})\cup(F\cap\hat{\phi}^{-1})=F\cap(\hat{\phi}\cup\hat{\phi}^{-1})=F\cap\text{const}^{V}=F$.
\end{enumerate}
\end{proof}

It is necessary to distinuish the above concept of selection from the selection of multifunction. Many theorems about multifunctions describe the selections in the second sense. 

\begin{definition}
Let $F\colon V\leadsto V$ be a multifunction. Then $\Pi_{v\in V}F(v)$ is called the set of selection of multifunction $F$. We say that $F$ has a selection iff $\Pi_{v\in V}F(v)\neq\emptyset$.

\end{definition}

In most of our cases each of our multifunctions is strict and therefore we always have a selection.

\begin{lemma}
If $F\colon V\leadsto V$ is a multifunction then $F$ has a selection iff $F$ is strict.
\end{lemma}

\begin{proof}
Fix a family of sets $(A_{t})\colon T\to P(X);t\mapsto A_{t}$. The axiom of choice is that $\Pi_{t\in T}A_{t}\neq\emptyset$ iff $\forall t\in T\colon A_{t}\neq\emptyset$. By definition $F$ is strict iff $\forall v\in V\colon F(v)\neq\emptyset$; or equivalently, $\Pi_{v\in V}F(v)\neq\emptyset$ i.e. $F$ has a selection.
\end{proof}

\subsection{Walks}
In our theory of multifunction we use the concepts derived from the graph theory. We do it all the time. However, we do not use the concept of the path because it is hard to express. Walks in graphs are easy to express as the finite pieces of trajectories of dynamical system $(F,V)$\cite{pilyugin}. In fact, we will only use the formula that there exists a walk between the vertices.   

\begin{definition}
Let $F\colon V\leadsto V$ be a multifunction, $u,w\in V$ and $n,m\in N$. Then$\colon$
\begin{itemize}
\item$Walk(F,V,n)=\{\alpha\in V^{\star}_{n+1}\mid\forall i\in\{1,\ldots,n\}\colon\alpha_{i}\in F(\alpha_{i+1})\}\subset V^{\star}_{n+1}$ is set of walks in $F$ with $n$ edges
\item$Walk(F,V\mid m)=\dot{\bigcup}_{n\in N}Walk(F,V,m\cdotp n)\subset V^{\star}$ is set of walks in $F$ with number of edges divisible by $m$
\item$Walk_{u\to w}(F,V,n)=\{\alpha\in Walk(F,V,n)\mid\alpha_{1}=u\land\alpha_{n+1}=w\}\subset Walk(F,V,n)$ is set of $(u,w)-$walks in $F$ with $n$ edges
\item$Walk_{u\to w}(F,V\mid m)=\dot{\bigcup}_{n\in N}Walk_{u\to w}(F,V,m\cdotp n)\subset V^{\star}$ is set of\newline$(u,w)-$walks in $F$ with number of edges divisible by $m$
\end{itemize}
\end{definition}
To prove the facts about walks we will do operations on words. To make calculations more flexible we need two new operations$\colon\hspace{1pt}\text{tail}$ and $\text{reverse}$\cite{graham}.
\begin{definition}
Let $V$ be a set and $i,j\in N$. Then$\colon$
\begin{itemize}
\item$\text{tail}\colon V^{\star}\to V^{\star};\alpha\mapsto\text{tail}(\alpha)=\beta$ iff $\text{tail}(\epsilon)=\epsilon$\newline$\forall i\in\{1,\ldots,\text{length}(\alpha)-1\}\colon\alpha_{i}=\beta_{i}$ and $\beta_{\text{length}(\alpha)}=\epsilon$
\item$\text{reverse}\colon V^{\star}\to V^{\star};\alpha\mapsto\text{reverse}(\alpha)=\beta$ iff $\text{reverse}(\epsilon)=\epsilon$\newline$\forall i\in\{1,\ldots,\text{length}(\alpha)\}\colon\beta_{i}=\alpha_{\text{length}(\alpha)-(i-1)}$
\end{itemize}
\end{definition} 
It is easily seen that the $\text{reverse}$ has the property $\text{reverse}\circ\text{reverse}=\text{id}_{V^{\star}}$ and therefore it is a bijection. We also need the following properties.
\begin{lemma}
If $V$ is a set then the equalities $\text{length}\circ\text{tail}=\text{length}\dot{-}1,\hspace{1pt}\text{length}\circ\text{reverse}=\text{length}$ and $\text{reverse}\circ\text{reverse}=\text{id}_{V^{\star}}$ hold. 
\end{lemma}

\begin{proof}
Fix $\alpha\in V^{\star}$ and $i\in\{1,\ldots,\text{length}(\alpha)\}$. 
\begin{enumerate}[(i)]
\item$\text{length}\circ\text{tail}(\alpha)=\text{length}(\text{tail}(\alpha))=\text{length}(\alpha)\dot{-}1$.
\item$\text{length}\circ\text{reverse}(\alpha)=\text{length}(\text{reverse}(\alpha))=\text{length}(\alpha)$.
\item$\text{reverse}\circ\text{reverse}(\alpha))_{i}=\text{reverse}(\text{reverse}(\alpha))_{i}=$\newline$=\text{reverse}(\alpha)_{\text{length}(\alpha)-(i-1)}=\alpha_{\text{length}(\alpha)-((\text{length}(\alpha)-(i-1))-1)}=\alpha_{i}$
\end{enumerate}
\end{proof}

Here we list some simple properties of walks.
\begin{lemma}
Let $F\colon V\leadsto V$ be a multifunction, $u,v,w\in V$ and $n,m\in N$. Then$\colon$
\begin{enumerate}[(i)]
\item$Walk(F,V,0)=V$
\item$Walk_{u\to u}(F,V,0)=\{u\}$
\item$Walk(F,V,1)=\{\beta\in V^{\star}_{2}\mid\beta_{1}\in F(\beta_{2})\}$
\item$Walk_{u\to w}(F,V,1)=\begin{cases}\emptyset$ iff $u\notin F(w)\\\{\text{concat}(u,w)\}$ iff $u\in F(w)\end{cases}$
\item$Walk(F^{-1},V,n)=\text{reverse}(Walk(F,V,n))$
\item$Walk_{u\to w}(F^{-1},V,n)=\text{reverse}(Walk_{w\to u}(F,V,n))$
\item if $Walk_{u\to v}(F,V,n)\neq\emptyset$ and $Walk_{v\to w}(F,V,m)\neq\emptyset$ then\newline$Walk_{u\to w}(F,V,n+m)\neq\emptyset$
\item if $Walk_{u\to w}(F,V,n)\neq\emptyset$ then there exist $k\in\{2,\ldots,n\}$ and $v\in V$ such that $Walk_{u\to v}(F,V,k-1)\neq\emptyset$ and $Walk_{v\to w}(F,V,n-(k-1))\neq\emptyset$, assuming $n\ge 2$
\item$Walk_{u\to w}(F,V,n)\neq\emptyset$ iff $\exists v\in V\colon\exists k\in\{2,\ldots,n\}\colon Walk_{v\to w}(F,$\newline$,V,n-(k-1))\neq\emptyset\land Walk_{v\to w}(F,V,k-1)\neq\emptyset$, assuming $n\ge 2$
\end{enumerate}
\end{lemma}

\begin{proof}
Fix $\alpha\in V^{\star}_{1}\simeq V,\hspace{1pt}\gamma\in V^{\star}$. 
\begin{enumerate}[(i)]
\item$\alpha\in Walk(F,V,0)$ iff $\alpha_{i}\in F(\alpha_{i+1})$ for each $i\in\{1,0\}=\emptyset$. This is always true and so $Walk(F,V,0)=V$.
\item$Walk_{u\to u}(F,V,0)=\{\alpha\in V^{\star}_{1}\mid\alpha_{1}=u\}=\{\alpha\in V\mid\alpha=u\}=\{u\}$.
\item$Walk(F,V,1)=\{\beta\in V^{\star}_{2}\mid\forall i\in\{1\}\colon\beta_{i}\in F(\beta_{i+1})\}=\{\beta\in V^{\star}_{2}\mid\beta_{1}\in F(\beta_{2})\}$.
\item$Walk_{u\to w}(F,V,1)=\{\beta\in Walk(F,V,1)\mid u=\beta_{1}\land\beta_{2}=w\}=\{\beta\in V^{\star}_{2}\mid\beta_{1}\in F(\beta_{2})\land\beta_{1}=u\land\beta_{2}=w\}=\begin{cases}\{\text{concat}(u,w)\}$ iff $u\in F(w)\\\emptyset$ iff $u\notin F(w)\end{cases}$.
\item If $\gamma\in V^{\star}_{n+1}$ then $\text{reverse}(\gamma)\in V^{\star}_{n+1}$. Therefore $Walk(F^{-1},V,n)=\{\gamma\in V^{\star}_{n+1}\mid\forall i\in\{1,\ldots,n\}\colon\gamma_{i}\in F^{-1}(\gamma_{i+1})\}=\{\gamma\in V^{\star}_{n+1}\mid\forall i\in\{1,\ldots,n\}\colon\gamma_{i+1}\in F(\gamma_{i})\}=\{\gamma\in V^{\star}_{n+1}\mid\forall i\in\{1,\ldots,n\}\colon\text{reverse}(\gamma)_{i}=\gamma_{\text{length}(\gamma)-(i-1)}=\gamma_{(n+1)-(i-1)}\in F(\gamma_{(n+1)-i})=F(\gamma_{\text{length}(\gamma)-((i+1)-1)})=F(\text{reverse}(\gamma)_{i+1})\}=\{\gamma\in V^{\star}_{n+1}\mid\text{reverse}(\gamma)\in Walk(F,V,n)\}=\text{reverse}(Walk(F,V,n))$
\item$Walk_{u\to w}(F^{-1},V,n)=\{\gamma\in Walk(F^{-1},V,n)\mid\gamma_{1}=u\land\gamma_{n+1}=w\}=\{\gamma\in V^{\star}_{n+1}\mid\gamma\in\text{reverse}(Walk(F,V,n))\land\text{reverse}(\gamma)_{n+1}=\gamma_{1}=u\land\text{reverse}(\gamma)_{1}=\gamma_{n+1}=w\}=\{\gamma\in V^{\star}_{n+1}\mid\text{reverse}(\gamma)\in Walk(F,V,n)\land\text{reverse}(\gamma)_{n+1}=u\land\text{reverse}(\gamma)_{1}=w\}=\text{reverse}(Walk_{w\to u}(F,V,n))$.
\item Assume that $Walk_{u\to v}(F,V,n)\neq\emptyset$ and $Walk_{v\to w}(F,V,m)\neq\emptyset$. Then there exist $\alpha\in V^{\star}_{n+1}$ and $\beta\in V^{\star}_{m+1}$ such that $\alpha\in Walk_{u\to v}(F,V,n)$ and $\beta\in Walk_{v\to w}(F,V,m)$. We may say equivalently that there exist $\alpha\in V^{\star}_{n+1}$ and $\beta\in V^{\star}_{m+1}$ such that $\alpha_{1}=u,\hspace{1pt}\alpha_{n+1}=v,\hspace{1pt}\beta_{1}=v,\hspace{1pt}\beta_{m+1}=w$ and $\alpha_{i}\in F(\alpha_{i+1}),\hspace{1pt}\beta_{j}\in F(\beta_{j+1})$ for each $i\in\{1,\ldots,n\}$ and $j\in\{1,\ldots,m\}$. Then $\text{concat}(\text{tail}(\alpha),\beta)\in V^{\star}_{n+m+1}$ and for every $k\in\{1,\ldots,n+m+1\}$\newline$\text{concat}(\text{tail}(\alpha),\beta)_{k}=\begin{cases}\alpha_{k}$ iff $k\in\{1,\ldots,n\}\\\beta_{k-n}$ iff $k\in\{n+1,\ldots,n+m+1\}\end{cases}$. Then\newline$\text{concat}(\text{tail}(\alpha),\beta)_{1}=\alpha_{1}=u$ and $\text{concat}(\text{tail}(\alpha),\beta)_{n+m+1}=\beta_{m+1}=w$ and $\text{concat}(\text{tail}(\alpha),\beta)_{k}\in F(\text{concat}(\text{tail}(\alpha),\beta)_{k+1})$ for each\newline $k\in\{1,\ldots,n+m\}$. Thus $\text{concat}(\text{tail}(\alpha),\beta)\in Walk_{u\to w}(F,V,n+m)$ i.e. $Walk_{u\to w}(F,V,n+m)\neq\emptyset$.
\item Assume that $Walk_{u\to w}(F,V,n)\neq\emptyset$ for some $n\ge 2$. Then there exists $\alpha\in V^{\star}_{n+1}$ such that $\alpha\in Walk_{u\to w}(F,V,n)$. We may say equivalently that there exists $\alpha\in V^{\star}_{n+1}$ such that $\alpha_{1}=u,\hspace{1pt}\alpha_{n+1}=w$ and $\alpha_{i}\in F(\alpha_{i+1})$ for each $i\in\{1,\ldots,n\}$. Then there exist $\alpha\in V^{\star}_{n+1},\hspace{1pt}k\in\{2,\ldots,n\}$ and $v\in V$ such that $\alpha_{1}=u,\hspace{1pt}\alpha_{k}=v,\alpha_{n+1}=w$ and $\alpha_{i}\in F(\alpha_{i+1}),\alpha_{j}\in F(\alpha_{j+1})$ for each $i\in\{1,\ldots,k-1\}$ and $j\in\{k,\ldots,n\}$. Therefore there exist $\beta\in V^{\star}_{k},\gamma\in V^{\star}_{n-k+2},k\in\{2,\ldots,n\}$ and $v\in V$ such that $u\alpha_{2}\ldots\alpha_{k-1}v=\beta\in Walk_{u\to v}(F,V,k-1)$ and $v\alpha_{k+1}\ldots\alpha_{n}w=\gamma\in Walk_{v\to w}(F,V,n-(k-1))$. Thus $\exists k\in\{2,\ldots,n\}\colon\exists v\in V\colon Walk_{u\to v}(F,V,k-1)\neq\emptyset\land Walk_{v\to w}(F,V,n-(k-1))\neq\emptyset$.
\item The last two implications give equivalence.
\end{enumerate}
\end{proof}

Recall the definition of walks in simple graphs.
\begin{definition}\cite{bollobas,diestel,jung}
Let $G=(V,E)$ be a simple graph, $u,w\in V$ and $n,m\in N$. Then$\colon$
\begin{itemize}
\item$Walk(G,n)=\{\alpha\in V^{\star}_{n+1}\mid\forall i\in\{1,\ldots,n\}\colon\{\alpha_{i},\alpha_{i+1}\}\in E\}\subset V^{\star}_{n+1}$ is set of walks with $n$ edges
\item$Walk(G\mid m)=\bigcup_{n\in N}Walk(G,m\cdotp n)$ is set of walks with number of edges divisible by $m$
\item$Walk_{u-w}(G,n)=\{\alpha\in Walk(G,n)\mid\alpha_{1}=u\land\alpha_{n+1}=w\}\subset Walk(G,n)$ is set of $\{u,w\}-$walks with $n$ edges
\item$Walk_{u-w}(G\mid m)=\bigcup_{n\in N}Walk_{u-w}(G,m\cdotp n)$ is set of $\{u,w\}-$walks with number of edges divisible by $m$
\end{itemize}

\end{definition}
Walks in simple graph multifunctions are walks in simple graphs.

\begin{lemma}
Let $F\colon V\leadsto V$ be a multifunction, $u,w\in V$ and $n,m\in N$. Then$\colon$
\begin{enumerate}[(i)]
\item$Walk(F,V,n)=Walk((V,\tilde{E_{F}}),n)$ 
\item$Walk_{u\to w}(F,V,n)=Walk_{u-w}((V,\tilde{E_{F}}),n)$  
\item$Walk(F,V\mid m)=Walk((V,\tilde{E_{F}})\mid m)$ 
\item$Walk_{u\to w}(F,V\mid m)=Walk_{u-w}((V,\tilde{E_{F}})\mid m)$  
\end{enumerate}

\end{lemma}

\begin{proof}
\begin{enumerate}[(i)]
\item$Walk(F,V,n)=\{\alpha\in V^{\star}_{n+1}\mid\forall i\in\{1,\ldots,n\}\colon\alpha_{i}\in F(\alpha_{i+1})\}=\{\alpha\in V^{\star}_{n+1}\mid\forall i\in\{1,\ldots,n\}\colon\{\alpha_{i},\alpha_{i+1}\}\in\tilde{E_{F}}\}=Walk((V,\tilde{E_{F}}),n)$.
\item$Walk_{u\to w}(F,V,n)=\{\alpha\in Walk(F,V,n)\mid\alpha_{1}=u\land\alpha_{n+1}=w\}=\{\alpha\in Walk((V,\tilde{E_{F}}),n)\mid\alpha_{1}=u\land\alpha_{n+1}=w\}=Walk_{u-w}((V,\tilde{E_{F}}),n)$.
\item$Walk(F,V\mid m)=\dot{\bigcup}_{n\in N}Walk(F,V,m\cdotp n)=$\newline$=\dot{\bigcup}_{n\in N}Walk((V,\tilde{E_{F}}),m\cdotp n)=Walk((V,\tilde{E_{F}})\mid m)$.
\item$Walk_{u\to w}(F,V\mid m)=\dot{\bigcup}_{n\in N}Walk_{u\to w}(F,V,m\cdotp n)=$\newline$=\dot{\bigcup}_{n\in N}Walk_{u-w}((V,\tilde{E_{F}}),m\cdotp n)=Walk_{u-w}((V,\tilde{E_{F}})\mid m)$.
\end{enumerate}
\end{proof}

\subsection{Bipartite graphs and independent sets}

We will use the concept of independent sets to define bipartite graphs and cliques.

\begin{definition}
Let $F\colon V\leadsto V$ be a multifunction and $U,W\in P(V)-\{\emptyset\}$. Then$\colon$
\begin{itemize}
\item$Ind(F,V)=\{U\in P(V)\mid U\cap F_{-}(U)=\emptyset\}\subset P(V)$ is the set of independent sets of $(F,V)$
\item$Clique(F,V)=Ind(F^{c},V)$ is the set of cliques of $(F,V)$
\item$F$ is $(U,W)-$bipartite iff $U,W\in Ind(F,V)$ and $V=U\dot{\cup}W$
\item$F$ is bipartite iff there exist $U,W\in P(V)-\{\emptyset\}$ such that $F$ is $(U,W)-$bipartite
\end{itemize}

\end{definition}
It is easily seen that for every multifunctions $F,G\colon V\leadsto V$ if $F\subset G$ then $Ind(G,V)\subset Ind(F,V)$ and if $F\subset G$ then $Clique(F,V)\subset Clique(G,V)$. 

We will need the following simple fact.
\begin{lemma}
Let $F\colon V\leadsto V$ be a multifunction and $U,W\in P(V)$. Then$\colon$
\begin{enumerate}[(i)]

\item$\exists(u,w)\in U\times W\colon u\in F(w)$ iff $W\cap F_{-}(U)\neq\emptyset$
\item$\forall(u,w)\in U\times W\colon u\in F(w)$ iff $W\cap(F^{c})_{-}(U)=\emptyset$
\item$U\in Clique(F,V)$ iff $\forall u,w\in U\colon u\in F(w)$
\item$U\in Ind(F,V)$ iff $\forall u,w\in U\colon u\notin F(w)$
\item$V\in Clique(F,V)$ iff $F=\text{const}^{V}$
\end{enumerate}

\end{lemma}

\begin{proof}
\begin{enumerate}[(i)]
\item$\exists(u,w)\in V\times W\colon u\in F(w)$ iff $F_{-}(U)\cap W\neq\emptyset$ because $U,W\subset V$.
\item$\forall (u,w)\in U\times W\colon u\in F(w)$ iff not $\exists(u,w)\in U\times W\colon u\notin F(w)$; or equivalently, not $\exists(u,w)\in U\times W\colon u\in F^{c}(w)$ i.e. $W\cap(F^{c})_{-}(U)=\emptyset$.
\item$U\in Clique(F,V)=Ind(F^{c},V)$ iff $U\cap(F^{c})_{-}(U)=\emptyset$; or equivalently, $\forall u,w\in U\colon u\in F(w)$.
\item$U\in Ind(F,V)=Clique(F^{c},V)$ iff $\forall u,w\in V\colon u\in F^{c}(w)$ i.e. $\forall u,w\in V\colon u\notin F(w)$.
\item$V\in Clique(F,V)$ iff $\forall u,w\in V\colon u\in F(w)$; or equivalently, $\forall w\in V\colon F(w)=V$ i.e. $F=\text{const}^{V}$. 
\end{enumerate}
\end{proof}
\newpage
We verify these definitions using simple graphs.
\begin{definition}\cite{bollobas,diestel,jung}
Let $G=(V,E)$ be a simple graph and $U,W\in P(V)-\{\emptyset\}$. Then$\colon$
\begin{itemize}
\item$Ind(G)=\{U\in P(V)\mid\forall u_{1},u_{2}\in U\colon\{u_{1},u_{2}\}\notin E\}$
\item$Clique(G)=\{U\in P(V)\mid\forall u_{1},u_{2}\in U\colon\{u_{1},u_{2}\}\in E\}$
\item$G$ is $(U,W)-$bipartite iff $U,W\in Ind(G)$ and $V=U\dot{\cup}W$
\item$G$ is bipartite iff there exist $U,W\in P(V)-\{\emptyset\}$ such that $G$ is $(U,W)-$bipartite
\end{itemize}
\end{definition}

\begin{lemma}
Let $F\colon V\leadsto V$ be a simple graph multifunction. Then$\colon$
\begin{itemize}
\item$Ind(F,V)=Ind((V,\tilde{E_{F}}))$
\item$Clique(F,V)=Clique((V,\tilde{E_{F}}))$
\item$F$ is bipartite iff $\tilde{G}=(V,\tilde{E_{F}})$ is bipartite

\end{itemize}

\end{lemma}

\begin{proof}
Fix $U\subset V$. 
\begin{enumerate}[(i)]
\item$U\in Ind(F,V)$ iff $\forall u_{1},u_{2}\in U\colon u_{1}\notin F(u_{2})$; or equivalently, $\forall u_{1},u_{2}\in U\colon\{u_{1},u_{2}\}\notin\tilde{E_{F}}$ i.e. $U\in Ind((V,\tilde{E_{F}}))$.
\item$U\in Clique(F,V)$ iff $\forall u_{1},u_{2}\in U\colon u_{1}\in F(u_{2})$; or equivalently, $\forall u_{1},u_{2}\in U\colon\{u_{1},u_{2}\}\in\tilde{E_{F}}$ i.e. $U\in Clique((V,\tilde{E_{F}}))$. 
\item$F$ is bipartite iff there exist $U,W\in P(V)-\{\emptyset\}$ such that $U,W\in Ind(F,V)=Ind((V,\tilde{E_{F}}))$ and $V=U\dot{\cup}W$ i.e. $\tilde{G_{F}}=(V,\tilde{E_{F}})$ is bipartite.

\end{enumerate}
\end{proof}

In the next part we will consider the graphs without isolated vertices. On this assumption, we have the following characterization of bipartite graphs in language of multifunctions.
\begin{lemma}
Let $F\colon V\leadsto V$ be a strict multifunction and $U,W\in P(V)-\{\emptyset\}$. Then $F$ is $(U,W)-$bipartite iff $V=U\dot{\cup}W$ and $F_{-}(U)=W$ and $F_{-}(W)=U$.
\end{lemma}

\begin{proof}
Assume that $F$ is bipartite. Then $V=U\dot{\cup}W$ and $F_{-}(U)\cap U=\emptyset=F_{-}(W)\cap W$. Therefore $V=U\dot{\cup}W$ and $F_{-}(U)\subset U^{c}=W$ and $F_{-}(W)\subset W^{c}=U$. Assume that $W\not\subset F_{-}(U)$ i.e. $W\cap F_{-}(U)^{c}\neq\emptyset$. Then there exists $v\in V$ such that $v\in W\cap F_{-}(U)^{c}$; or equivalently, there exists $v\in V$ such that $v\in W$ and $F(v)\cap U=\emptyset$. We may say equivalently that there exists $v\in W$ and $F(v)\subset U^{c}=W$ i.e. $\emptyset\neq W\cap F_{+}(W)\subset W\cap F_{-}(W)$, assuming that $F$ is strict. Thus $W\cap F_{-}(W)\neq\emptyset$, contradiction. Hence $W\subset F_{-}(U)\subset W$ i.e. $F_{-}(U)=W$ and for the same reason $F_{-}(W)=U$.

Assume that $V=U\dot{\cup}W,F_{-}(U)=W$ and $F_{-}(W)=U$. Then $F_{-}(U)\subset W=U^{c}$ and $F_{-}(W)\subset U=W^{c}$; or equivalently, $F_{-}(U)\cap U=\emptyset$ and $F_{-}(W)\cap W=\emptyset$ i.e. $U,W\in Ind(F,V)$. Then, by definition, $F$\newline is $(U,W)-$bipartite.
\end{proof}

Here is one of the mostly used characterizations of bipartite graphs.
\begin{theorem}\cite{bollobas}[K\"onig 1916]
The simple graph $G=(V,E)$ is bipartite iff $G=(V,E)$ has no cycle of odd length.  
\end{theorem}
We may say equivalently that a simple graph $G=(V,E)$ is not bipartite iff $G=(V,E)$ has a cycle of odd length. Furthermore, in proof of this theorem, we do not use the finiteness of graph $G$. We will prove the K\"onig theorem for multifunctions that says that if $F\colon V\leadsto V$ $F$ is a connected simple graph multifunction then $F$ is bipartite iff for every $n\in N$ the multifunction $F^{(2\cdotp n+1)\cup}$ is loopless.

We will prove another characterization of bipartite graphs that says that if $F\colon V\leadsto V$ is a connected simple graph multifunction then $F$ is bipartite iff $F^{2\cup}$ is disconnected. Furthermore $F$ is bipartite iff for every $n\in N$ the multifunction $F^{(2\cdotp n)\cup}$ is disconnected. 

Now we are introducing our main tool.

\section{Iterations of multifunctions}

We have iterations of multifunctions in similar way to dynamical systems. We use only iterations of $F_{\cup}$ and $F_{-}$. The reason for this is the fact that $F_{\cap}$ and $F_{+}$ are expressed by $F_{\cup}$ and $F_{-}$.

\begin{definition}\cite{bergetop}
Let $F\colon V\leadsto V$ be a multifunction and $n,m\in Z$. Then$\colon$
\begin{itemize}
\item$F_{\cup}^{n}\colon P(V)\to P(V);A\mapsto F_{\cup}^{n}(A)=\begin{cases}A$ iff $n=0\\F_{\cup}(F_{\cup}^{n-1}(A))$ iff $n>0\\(F^{-1})_{\cup}^{-n}(A)$ iff $n<0\end{cases}$
\item$F_{-}^{n}\colon P(V)\to P(V);A\mapsto F_{-}^{n}(A)=\begin{cases}A$ iff $n=0\\F_{-}(F_{-}^{n-1}(A))$ iff $n>0\\(F^{-1})_{-}^{-n}(A)$ iff $n<0\end{cases}$
\item$F^{n\cup}\colon V\to P(V);v\mapsto F^{n\cup}(v)=\begin{cases}\{v\}$ iff $n=0\\F_{\cup}(F^{(n-1)\cup}(v))$ iff $n>0\\(F^{-1})^{(-n)\cup}(v)$ iff $n<0\end{cases}$
\item$F^{n-}\colon V\to P(V);v\mapsto F^{n-}(v)=\begin{cases}\{v\}$ iff $n=0\\F_{-}(F^{(n-1)-}(v))$ iff $n>0\\(F^{-1})^{(-n)-}(v)$ iff $n<0\end{cases}$
\item$F^{+\infty\cup}_{\mid m}=\bigcup_{n\in N}F^{(m\cdotp n)\cup}$
\item$F^{-\infty\cup}_{\mid m}=\bigcup_{n\in-N}F^{(m\cdotp n)\cup}$
\item$F^{\infty\cup}_{\mid m}=F^{+\infty\cup}_{\mid m}\cup F^{-\infty\cup}_{\mid m}$
\item$F^{+\infty-}_{\mid m}=\bigcup_{n\in N}F^{(m\cdotp n)-}$
\item$F^{-\infty-}_{\mid m}=\bigcup_{n\in-N}F^{(m\cdotp n)-}$
\item$F^{\infty-}_{\mid m}=F^{+\infty-}_{\mid m}\cup F^{-\infty-}_{\mid m}$
\end{itemize}

\end{definition}

Inclusion of multifunctions is preserved by iterations.
\begin{lemma}
If $F,G\colon V\leadsto V$ are multifunctions such that $F\subset G$ then $F^{n\cup}\subset G^{n\cup}$ and $F^{+\infty\cup}_{\mid n}\subset G^{+\infty\cup}_{\mid n}$ for each $n\in N$.
\end{lemma}
\begin{proof}
First inclusion we prove inductively. Fix $v\in V$. Assume that $F\subset G$. For $n=1$ we calculate that $F^{1\cup}=F\subset G=G^{1\cup}$. For $n=2$ we calculate that $F^{2\cup}(v)=F_{\cup}(F(v))\subset F_{\cup}(G(v))\subset G_{\cup}(G(v))=G^{2\cup}(v)$. Assume that $F^{n\cup}(v)\subset G^{n\cup}(v)$. Then $F^{(n+1)\cup}(v)=F_{\cup}(F^{n\cup}(v))\subset F_{\cup}(G^{n\cup}(v))=G_{\cup}(G^{n\cup}(v))=G^{(n+1)\cup}(v)$.

Therefore $F^{+\infty\cup}_{\mid n}=\bigcup_{k\in N}F^{(k\cdotp n)\cup}\subset\bigcup_{k\in N}G^{(k\cdotp n)\cup}=G^{+\infty\cup}_{\mid n}$.
\end{proof}

The graphs we study here do not have isolated vertices. This property is preserved by iteration.
\begin{lemma}
Let $F\colon V\leadsto V$ be a multifunction. Then$\colon$
\begin{enumerate}[(i)]
\item if $F$ is not strict then $\forall n\in N\colon F^{n\cup}$ is not strict
\item if $F$ is strict then $\forall n\in N\colon F^{n\cup}$ is strict
\item$F$ is strict iff $\forall n\in N\colon F^{n\cup}$ is strict
\item$F$ is strict iff $\exists n\in N\colon F^{n\cup}$ is strict
\item$F$ is strict iff $F^{+\infty\cup}_{\mid 1}$ is strict
 
\end{enumerate}
\end{lemma} 

\begin{proof}
 
\begin{enumerate}[(i)]
\item Assume that $F$ is not strict. Then there exists $v\in V$ such that $F(v)=\emptyset$. Then $F^{2\cup}(v)=F_{\cup}(F^{1\cup}(v))=F_{\cup}(F(v))=F_{\cup}(\emptyset)=\emptyset$. Assume inductively that $F^{n\cup}(v)=\emptyset$ is true. Then $F^{(n+1)\cup}(v)=F_{\cup}(F^{n\cup}(v))=F_{\cup}(\emptyset)=\emptyset$. Hence $F^{n\cup}$ is not strict for any $n\in N$.
\item Assume that $F$ is strict. Fix $v\in V$. Then $F(v)\neq\emptyset$ and $F^{2\cup}(v)=F_{\cup}(F^{1\cup}(v))=F_{\cup}(F(v))\neq\emptyset$. Assume inductively that $F^{n\cup}(v)\neq\emptyset$ is true. Then $F^{(n+1)\cup}(v)=F_{\cup}(F^{n\cup}(v))\neq\emptyset$. Hence $F^{n\cup}$ is strict for any $n\in N$.  
\item if $\exists n\in N\colon F^{n\cup}$ is strict then $F$ is strict. But if $F$ is strict then $\forall n\in N\colon F^{n\cup}$ is strict. Notice that if $\forall n\in N\colon F^{n\cup}$ is strict then $\exists n\in N\colon F^{n\cup}$ is strict. Therefore $F$ is strict iff $\forall n\in N\colon F^{n\cup}$ is strict.
\item From the above $F$ is strict iff $\exists n\in N\colon F^{n\cup}$ is strict.
\item$F$ is strict iff $\exists n\in N\colon F^{n\cup}$ is strict; or equivalently, $F^{+\infty\cup}_{\mid 1}=\bigcup_{n\in N}F^{n\cup}$ is strict.
  
\end{enumerate}

\end{proof}

We explain how these iterations are related together.
\begin{lemma}
Let $F\colon V\leadsto V$ be a multifunction and $n\in N$. Then$\colon$

\begin{enumerate}[(i)]
\item$F^{n\cup}=F^{n}_{\cup}\circ\{\cdot\}$
\item$F^{n-}=F^{n}_{-}\circ\{\cdot\}$
\item$F^{n}_{-}=(F^{-1})^{n}_{\cup}$
\item$F^{n-}=(F^{-1})^{n\cup}$
\end{enumerate}

\end{lemma}

\begin{proof}
Almost all equalities listed here we will prove inductively. Fix $w\in V,A\in P(V),n\in N$. We write the first inductive step for all three cases. For $n=0$ we calculate $F^{0\cup}(w)=\{w\}=F^{0}_{\cup}(\{w\})=(F^{0}_{\cup}\circ\{\cdot\})(w)$ and $F^{0-}(w)=\{w\}=F^{0}_{-}(\{w\})=(F^{0}_{-}\circ\{\cdot\})(w)$ and $F^{0}_{-}(A)=A=(F^{-1})^{0}_{-}$. For $n=1$ we calculate $F^{1\cup}(w)=F_{\cup}(F^{0\cup}(w))=F_{\cup}(\{w\})=(F_{\cup}\circ\{\cdot\})(w)=F(w)=(F_{\cup}\circ\{\cdot\})(w)=F_{\cup}(\{w\})=F_{\cup}(F_{\cup}^{0}(\{w\}))=F^{1}_{\cup}(\{w\})=(F^{1}_{\cup}\circ\{\cdot\})(w)$
and $F^{1-}(w)=F_{-}(F^{0-}(w))=F_{-}(\{w\})=(F_{-}\circ\{\cdot\})(w)=F^{-1}(w)=(F_{-}\circ\{\cdot\})(w)=F_{-}(\{w\})=F_{-}(F_{-}^{0}(\{w\}))=F^{1}_{-}(\{w\})=(F^{1}_{-}\circ\{\cdot\})(w)$
and $F^{1}_{-}(A)=F_{-}(F^{0}_{-}(A))=F_{-}(A)=F^{-1}_{\cup}(A)=F^{-1}_{\cup}((F^{-1})^{0}_{\cup}(A))=(F^{-1})^{1}_{\cup}(A)$.
\begin{enumerate}[(i)]
\item Assume that $F^{n\cup}=F^{n}_{\cup}\circ\{\cdot\}$ is true. Then $F^{(n+1)\cup}(w)=F_{\cup}(F^{n\cup}(w))=F_{\cup}((F^{n}_{\cup}\circ\{\cdot\})(w))=F_{\cup}(F^{n}_{\cup}(\{w\}))=F^{n+1}_{\cup}(\{w\})=(F^{n+1}_{\cup}\circ\{\cdotp\})(w)$. Thus $\forall n\in N\colon F^{n\cup}=F^{n}_{\cup}\circ\{\cdot\}$ is true.
\item Assume that $F^{n-}=F^{n}_{-}\circ\{\cdot\}$ is true. Then $F^{(n+1)-}(w)=F_{-}(F^{n-}(w))=F_{-}((F^{n}_{-}\circ\{\cdot\})(w))=F_{-}(F^{n}_{-}(\{w\}))=F^{n+1}_{-}(\{w\})=(F^{n+1}_{-}\circ\{\cdotp\})(w)$. Thus $\forall n\in N\colon F^{n-}=F^{n}_{-}\circ\{\cdot\}$ is true.
\item Assume that $F^{n}_{-}=(F^{-1})^{n}_{\cup}$ is true. Then $F^{n+1}_{-}(A)=F_{-}(F^{n}_{-}(A))=$\newline$=F_{-}((F^{-1})^{n}_{\cup}(A))=(F^{-1})_{\cup}((F^{-1})^{n}_{\cup}(A))=(F^{-1})^{n+1}_{\cup}(A)$. Thus\newline$\forall n\in N\colon F^{n}_{-}=(F^{-1})^{n}_{\cup}$ is true.
\item$F^{n-}=F^{n}_{-}\circ\{\cdot\}=(F^{-1})^{n}_{\cup}\circ\{\cdot\}=(F^{-1})^{n\cup}$.
\end{enumerate}
\end{proof}
All these equations hold for the integers.
\begin{lemma}
Let $F\colon V\leadsto V$ be a multifunction and $n\in Z$. Then$\colon$

\begin{enumerate}[(i)]
\item$F^{n\cup}=F^{n}_{\cup}\circ\{\cdot\}$
\item$F^{n-}=F^{n}_{-}\circ\{\cdot\}$
\item$F^{n}_{-}=(F^{-1})^{n}_{\cup}$
\item$F^{n-}=(F^{-1})^{n\cup}$
\item$F^{n}_{\cup}=F^{n\cup}_{\cup}$
\end{enumerate}

\end{lemma}

\begin{proof}
Everything is proven for $n\in N$. Fix $A\subset V,n\in-N$. Then$\colon$
\begin{enumerate}[(i)]
\item$F^{n\cup}=(F^{-1})^{(-n)\cup}=(F^{-1})^{-n}_{\cup}\circ\{\cdot\}=F^{n}_{\cup}\circ\{\cdot\}$
\item$F^{n-}=(F^{-1})^{(-n)-}=(F^{-1})^{-n}_{-}\circ\{\cdot\}=F^{n}_{-}\circ\{\cdot\}$
\item$F^{n}_{-}=(F^{-1})_{-}^{-n}=F^{-n}_{\cup}=(F^{-1})_{\cup}^{n}$
\item$F^{n-}=(F^{-1})^{(-n)-}=F^{(-n)\cup}=(F^{-1})^{n\cup}$
\item$F^{n}_{\cup}(A)=F^{n}_{\cup}(\bigcup_{a\in A}\{a\})=\bigcup_{a\in A}F^{n}_{\cup}(\{a\})=\bigcup_{a\in A}(F^{n}_{\cup}\circ\{\cdot\})(a)=$\newline$=\bigcup_{a\in A}F^{n\cup}(a)=F^{n\cup}_{\cup}(A)$
\end{enumerate}
\end{proof}

For this reason we have the following equality for infinity.

\begin{lemma}
Let $F\colon V\leadsto V$ be a multifunction and $m\in N$. Then$\colon$
\begin{enumerate}[(i)]
\item$F^{+\infty-}_{\mid m}=(F^{-1})^{+\infty\cup}_{\mid m}$
\item$F^{-\infty-}_{\mid m}=(F^{-1})^{-\infty\cup}_{\mid m}$
\item$F^{\infty-}_{\mid m}=(F^{-1})^{\infty\cup}_{\mid m}$

\end{enumerate}

\end{lemma}

\begin{proof}
It is easily seen that$\colon$
\begin{enumerate}[(i)]
\item$F^{+\infty-}_{\mid m}=\bigcup_{n\in N}F^{(m\cdotp n)-}=\bigcup_{n\in N}(F^{-1})^{(m\cdotp n)\cup}=(F^{-1})^{+\infty\cup}_{\mid m}$
\item$F^{-\infty-}_{\mid m}=\bigcup_{n\in-N}F^{(m\cdotp n)-}=\bigcup_{n\in-N}(F^{-1})^{(m\cdotp n)\cup}=(F^{-1})^{-\infty\cup}_{\mid m}$ \item$F^{\infty-}_{\mid m}=F^{+\infty-}_{\mid m}\cup F^{-\infty-}_{\mid m}=(F^{-1})^{+\infty\cup}_{\mid m}\cup(F^{-1})^{-\infty\cup}_{\mid m}=(F^{-1})^{\infty\cup}_{\mid m}$. 
\end{enumerate}
\end{proof}
These equations allow us to use only $F^{n\cup}$ and $F^{n}_{\cup}$. Here we list some simple properties of these iterations starting from the concept of walk. 

\begin{lemma}
If $F\colon V\leadsto V$ is a multifunction then $F^{n\cup}\colon V\to P(V);w\mapsto F^{n\cup}(w)=\{u\in V\mid Walk_{u\to w}(F,V,n)\neq\emptyset\}$ for each $n\in N$.

\end{lemma}

\begin{proof}
Fix $u,w\in V$. We prove by induction. For $n=0$ we calculate $u\in F^{0\cup}(w)=\{w\}=Walk_{w\to w}(F,V,0)$ iff $Walk_{u\to w}(F,V,0)\neq\emptyset$. For $n=1$ we calculate $u\in F^{1\cup}(w)=F(w)$ iff $Walk_{u\to w}(F,V,1)\neq\emptyset$. Assume that $u\in F^{n\cup}(w)\Leftrightarrow Walk_{u\to w}(F,V,n)\neq\emptyset$. Then $u\in F^{(n+1)\cup}(w)=F_{\cup}(F^{n\cup}(w))$ iff there exists $v\in V$ such that $v\in F^{n\cup}(w)$ and $u\in F(v)=F^{1\cup}(v)$; or equivalently, there exists $v\in V$ such that $Walk_{v\to w}(F,V,n)\neq\emptyset$ and $Walk_{u\to v}(F,V,1)\neq\emptyset$. We may say equivalently that $Walk_{u\to w}(F,V,n+1)\neq\emptyset$.  
\end{proof}

\begin{lemma}
Let $F\colon V\leadsto V$ be a multifunction and $n,m\in N$. Then$\colon$
\begin{enumerate}[(i)]
\item$F^{(n+m)\cup}=F_{\cup}^{n}\circ F^{m\cup}$
\item$F^{(-n)\cup}=(F^{-1})^{n\cup}=(F^{n\cup})^{-1}$
\item$F^{n\cup}_{-}=F^{n}_{-}$
\item$(F^{n\cup})^{m\cup}=F^{(n\cdotp m)\cup}$
\end{enumerate}
\end{lemma}

\begin{proof}
Fix $u,w\in V$.
\begin{enumerate}[(i)]
\item$u\in F_{\cup}^{n}(F^{m\cup}(w))=F_{\cup}^{n\cup}(F^{m\cup}(w))$ iff there exists $v\in V$ such that $v\in F^{m\cup}(w)$ and $u\in F^{n\cup}(v)$; or equivalently, there exists $v\in V$ such that $\emptyset\neq Walk_{v\to w}(F,V,m)$ and $\emptyset\neq Walk_{u\to v}(F,V,n)$. We may say equivalently that $\emptyset\neq Walk_{u\to w}(F,V,m+n)$ i.e. $u\in F^{(n+m)\cup}(w)$.
\item$F^{(-n)\cup}(w)=(F^{-1})^{n\cup}(w)=\{u\in V\mid Walk_{u\to w}(F^{-1},V,n)\neq\emptyset\}=\{u\in V\mid\text{reverse}(Walk_{w\to u}(F,V,n))\}=\{u\in V\mid Walk_{w\to u}(F,V,n)\neq\emptyset\}=(F^{n\cup})^{-1}(w)$.
\item$F^{n\cup}_{-}=(F^{n\cup})_{-}=(F^{n\cup})^{-1}_{\cup}=(F^{-1})^{n\cup}_{\cup}=(F^{-1})^{n}_{\cup}=F^{n}_{-}$.
\item We prove by induction. For $n=0$ we calculate $(F^{n\cup})^{0\cup}(w)=\{w\}=F^{0\cup}(w)=F^{(n\cdotp 0)\cup}(w)$. For $n=1$ we calculate $(F^{n\cup})^{1\cup}(w)=(F^{n\cup}_{\cup}\circ\{\cdot\})(w)=(F^{n}_{\cup}\circ\{\cdot\})(w)=F^{n\cup}(w)=F^{(n\cdotp 1)\cup}(w)$. For $n=2$ we calculate $(F^{n\cup})^{2\cup}(w)=F^{n\cup}_{\cup}((F^{n\cup})^{1\cup}(w))=F^{n\cup}_{\cup}(F^{n\cup}(w))=F^{n}_{\cup}(F^{n\cup}(w))=F^{(n+n)\cup}(w)=F^{(2\cdotp n)\cup}(w)$.
Assume that $(F^{n\cup})^{m\cup}=F^{(n\cdotp m)\cup}$ is true. Then $(F^{n\cup})^{(m+1)\cup}=(F^{n\cup})^{m}_{\cup}\circ(F^{n\cup})^{1\cup}=(F^{n\cup})^{m\cup}_{\cup}\circ F^{n\cup}=F^{(n\cdotp m)\cup}_{\cup}\circ F^{n\cup}=F^{n\cdotp m}_{\cup}\circ F^{n\cup}=F^{(n\cdotp m+n)\cup}=F^{(n\cdotp(m+1))\cup}$.\newline Thus $\forall m,n\in N\colon(F^{n\cup})^{m\cup}=F^{(n\cdotp m)\cup}$ is true.

\end{enumerate}
\end{proof}
Only the last of these equations is fulfilled for integers.
\begin{lemma}
If $F\colon V\leadsto V$ is a multifunction then $(F^{n\cup})^{m\cup}=F^{(n\cdotp m)\cup}$ for each $n,m\in Z$.

\end{lemma}

\begin{proof}
Fix $m,n\in N$. We proved that $(F^{n\cup})^{m\cup}=F^{(n\cdotp m)\cup}$ for any $n,m\in N$. Then $(F^{n\cup})^{(-m)\cup}=((F^{n\cup})^{-1})^{m\cup}=(F^{(-n)\cup})^{m\cup}=((F^{-1})^{n\cup})^{m\cup}=(F^{-1})^{(n\cdotp m)\cup}=F^{(-n\cdotp m)\cup}$ and $(F^{(-n)\cup})^{(-m)\cup}=F^{(-(-n)\cdotp m)\cup}=F^{(n\cdotp m)\cup}$. Thus $(F^{n\cup})^{m\cup}=F^{(n\cdotp m)\cup}$ for each $n,m\in Z$.
\end{proof}

The situation is simplified when we use only strict and undirected multifunctions. That is why it is necessary to consider only such multifunctions.
\begin{lemma}
Let $F\colon V\leadsto V$ be a multifunction. Then$\colon$
\begin{enumerate}[(i)]
\item$F$ is undirected iff $F^{n\cup}$ is undirected for each $n\in Z$
\item if $F$ is stirct and undirected then $F^{2\cup}$ is everywhereloop 
\item if $F$ is strict and undirected then $F^{n\cup}\subset F^{(n+2)\cup}$ for each $n\in N$
\item if $F$ is strict and undirected then $F^{(2\cdotp n)\cup}$ is everywhereloop for each $n\in N$ 
\item if $F$ is undirected then $F^{\infty\cup}_{\mid m}=F^{-\infty\cup}_{\mid m}=F^{+\infty\cup}_{\mid m}$ for each $m\in N$
\end{enumerate}

\end{lemma}

\begin{proof}
Fix $n\in Z$ and $u,v,w\in V$.
\begin{enumerate}[(i)]
\item If $F=F^{-1}$ then $(F^{n\cup})^{-1}=F^{(-n)\cup}=(F^{-1})^{n\cup}=F^{n\cup}$. If $F^{n\cup}=F^{(-n)\cup}$ for each $n\in Z$ then $F=F^{1\cup}=F^{(-1)\cup}=F^{-1}$ i.e. $F$ is undirected.
\item$v\in F^{2\cup}(v)=F_{\cup}(F^{1\cup}(v))=F_{\cup}(F(v))$ iff there exists $a\in V$ such that $a\in F(v)$ and $v\in F(a)$; or equivalently, there exists $a\in V$ such that $a\in F(v)$ and $a\in F^{-1}(v)=F(v)$, assuming that $F$ is undirected. We may say equivalently that there exists $a\in V$ such that $a\in F(v)$ i.e. $F(v)\neq\emptyset$, but $F$ is strict. 
\item If $u\in F^{n\cup}(w)$ then there exists $a\in V,a=w$ such that $u\in F^{n\cup}(a)$ and $a\in F^{2\cup}(w)$ i.e. $u\in\bigcup_{a\in F^{2\cup}(w)}F^{n\cup}(a)=F^{n\cup}_{\cup}(F^{2\cup}(w))=$\newline$=F^{n}_{\cup}(F^{2\cup}(w))=F^{(n+2)\cup}(w)$. Thus $F^{n\cup}(w)\subset F^{(n+2)\cup}(w)$ i.e. $F^{(2\cdotp n)\cup}$ is everywhereloop for each $n\in N$.
\item We prove by induction. First inductive step is true because $F^{2\cup}$ is everywhereloop. Assume that $\{\cdotp\}\subset F^{(2\cdotp n)\cup}$ is true i.e. $F^{(2\cdotp n)\cup}$ is everywhereloop. Then $\{\cdotp\}\subset F^{(2\cdotp n)\cup}\subset F^{(2\cdotp n+2)\cup}=F^{(2\cdotp(n+1))\cup}$ i.e. $F^{(2\cdotp(n+1))\cup}$ is everywhereloop.
\item$F$ is undirected iff $F^{n\cup}$ is undirected for each $n\in Z$ i.e. $F^{(-n)\cup}=(F^{n\cup})^{-1}=F^{n\cup}$ and so $F^{+\infty\cup}_{\mid m}=\bigcup_{n\in N}F^{(m\cdotp n)\cup}=\bigcup_{n\in-N}F^{((-n)\cdotp m)\cup}=\bigcup_{n\in-N}F^{(n\cdotp m)\cup}=F^{-\infty\cup}_{\mid m}$. Then $F^{\infty\cup}_{\mid m}=F^{+\infty\cup}_{\mid m}\cup F^{-\infty\cup}_{\mid m}=F^{+\infty\cup}_{\mid m}=F^{-\infty\cup}_{\mid m}$. 

\end{enumerate}
\end{proof}

Although we have the equality $(F^{n\cup})^{m\cup}=F^{(n\cdotp m)\cup}$ for each $n,m\in N$ we can't generalize it on infinity.  
\begin{lemma}
Let $F\colon V\leadsto V$ be an undirected multifunction and $m,k\in N$. Then$\colon$
\begin{enumerate}[(i)]
\item$F^{\infty\cup}_{\mid m}\colon V\to P(V);w\mapsto F^{\infty\cup}_{\mid m}(w)=\{u\in V\mid Walk_{u\to w}(F,V\mid m)\neq\emptyset\}$ is undirected, everywhereloop and transitive
\item$F^{\infty\cup}_{\mid-m}=F^{\infty-}_{\mid m}=F^{\infty\cup}_{\mid m}$
\item$(F^{m\cup})^{\infty\cup}_{\mid k}=F^{\infty\cup}_{\mid k\cdotp m}$
\item if $k\mid m$ then $F^{\infty\cup}_{\mid m}\subset F^{\infty\cup}_{\mid k}$ and $F^{m\cup}\subset F^{\infty\cup}_{\mid k}$
\item$F^{m\cup}\subset F^{\infty\cup}_{\mid m}\subset F^{\infty\cup}_{\mid 1}$
\end{enumerate}

\end{lemma}

\begin{proof}
Fix $u,v,w\in V$.
\begin{enumerate}[(i)]
\item$F^{\infty\cup}_{\mid m}(w)=\bigcup_{n\in N}F^{(m\cdotp n)\cup}(w)=\bigcup_{n\in N}\{u\in V\mid Walk_{u\to w}(F,V,m\cdotp n)\neq\emptyset\}=\{u\in V\mid\bigcup_{n\in N}Walk_{u\to w}(F,V,m\cdotp n)\neq\emptyset\}=$\newline$=\{u\in V\mid Walk_{u\to w}(F,V\mid m)\neq\emptyset\}$. But $F^{(2\cdotp m)\cup}$ is everywhereloop and $F^{(2\cdotp m)\cup}(w)\subset\bigcup_{n\in N}F^{(m\cdotp n)\cup}(w)=F^{\infty\cup}_{\mid m}(w)$, thus $F^{\infty\cup}_{\mid m}$ is everywhereloop. Assume that $u\in F^{\infty\cup}_{\mid m}(v)$ and $v\in F^{\infty\cup}_{\mid m}(w)$.\newline Then there exist $i,j\in N$ such that $u\in F^{(i\cdotp m)\cup}(v)$ and $v\in F^{(j\cdotp m)\cup}(w)$. Then there exist $i,j\in N$ and $a\in V,a=v$ such that $a\in F^{(j\cdotp m)\cup}(w)$ and $u\in F^{(i\cdotp m)\cup}(a)$; or equivalently, $u\in\bigcup_{a\in F^{(j\cdotp m)\cup}(w)}F^{(i\cdotp m)\cup}(a)=F^{(i\cdotp m)\cup}_{\cup}(F^{(j\cdotp m)\cup}(w))=F^{i\cdotp m}_{\cup}(F^{(j\cdotp m)\cup}(w))=F^{((i+j)\cdotp m)\cup}(w)\subset$\newline$\subset\bigcup_{n\in N}F^{(n\cdotp m)\cup}(w)=F^{\infty\cup}_{\mid m}(w)$ and then $u\in F^{\infty\cup}_{\mid m}(w)$, thus $F^{\infty\cup}_{\mid m}$ is transitive. $F^{\infty\cup}_{\mid m}$ is undirected because $(F^{\infty\cup}_{\mid m})^{-1}=(\bigcup_{n\in N}F^{(n\cdotp m)\cup})^{-1}=$\newline$=\bigcup_{n\in N}(F^{(m\cdotp n)\cup})^{-1}=\bigcup_{n\in N}(F^{-1})^{(m\cdotp n)\cup}=\bigcup_{n\in N}F^{(m\cdotp n)\cup}=F^{\infty\cup}_{\mid m}$.
\item$F^{\infty\cup}_{\mid-m}=\bigcup_{n\in N}F^{(-m\cdotp n)\cup}=\bigcup_{n\in N}(F^{-1})^{(m\cdotp n)\cup}=(F^{-1})^{\infty\cup}_{\mid m}=F^{\infty\cup}_{\mid m}=$\newline$=(F^{-1})^{\infty\cup}_{\mid m}=F^{\infty-}_{\mid m}$
\item$(F^{m\cup})^{\infty\cup}_{\mid k}=\bigcup_{n\in N}F^{(k\cdotp m\cdotp n)\cup}=F^{\infty\cup}_{\mid k\cdotp m}$.
\item Recall that $k\mid m$ iff there exists $l\in Z$ such that $m=k\cdotp l$. Then $F^{\infty\cup}_{\mid m}=\bigcup_{n\in N}F^{(n\cdotp m)\cup}=\bigcup_{n\in N}F^{(k\cdotp l\cdotp n)\cup}\subset\bigcup_{\bar{n}\in N}F^{(\bar{n}\cdotp k)\cup}=F^{\infty\cup}_{\mid k}$ and $F^{m\cup}=F^{(l\cdotp k)\cup}\subset\bigcup_{n\in N}F^{(n\cdotp k)\cup}=F^{\infty\cup}_{\mid k}$.
\item It follows from the simple fact that $m\mid m$ and $1\mid m$.
\end{enumerate}
\end{proof}

That is why we are accustomed to operating on undirected multifunctions. We are going to prove K\"onig theorem in the language of multifunctions. We finish this topic by making an observation about strict and bipartite multifunctions.

\begin{lemma}
Let $F\colon V\leadsto V$ be a strict undirected multifunction and $U,W\in P(V)-\{\emptyset\}$. Then $F$ is $(U,W)-$bipartite iff $V=U\dot{\cup}W$ and for every $n\in N$ the following equalities are fulfilled$\colon\hspace{1pt}F^{2\cdotp n}_{-}(U)=U,\hspace{1pt}F^{2\cdotp n}_{-}(W)=W,\hspace{1pt}F^{2\cdotp n+1}_{-}(U)=W$ and $F^{2\cdotp n+1}_{-}(W)=U$.
\end{lemma}

\begin{proof}
Assume that $F$ is $(U,W)-$bipartite. Then $V=U\dot{\cup}W$ and $F_{-}(U)=W$ and $F_{-}(W)=U$. We prove the iteration inductively. For $n=0$ we calculate $F^{2\cdotp 0}_{-}(U)=F^{0}_{-}(U)=U$ and $F^{2\cdotp 0}_{-}(W)=F^{0}_{-}(W)=W$.
For $n=1$ we calculate $F^{2\cdotp 1}_{-}(U)=F^{2}_{-}(U)=F_{-}(F^{1}_{-}(U))=F_{-}(F_{-}(U))=F_{-}(W)=U$ and $F^{2\cdotp 1}_{-}(W)=F^{2}_{-}(W)=F_{-}(F^{1}_{-}(W))=F_{-}(F_{-}(W))=F_{-}(U)=W$.

Assume that $F^{2\cdotp n}_{-}(U)=U$. Then $F^{2\cdotp(n+1)}_{-}(U)=F^{2\cdotp n+2}_{-}(U)=$\newline$=F_{-}(F^{2\cdotp n+1}_{-}(U))=F_{-}(F_{-}(F^{2\cdotp n}_{-}(U)))=F_{-}(F_{-}(U))=F_{-}(W)=U$. Therefore $F^{2\cdotp n}_{-}(U)=U$ for each $n\in N$.

Assume that $F^{2\cdotp n}_{-}(W)=W$. Then $F^{2\cdotp(n+1)}_{-}(W)=F^{2\cdotp n+2}_{-}(W)=$\newline$=F_{-}(F^{2\cdotp n+1}_{-}(W))=F_{-}(F_{-}(F^{2\cdotp n}_{-}(W)))=F_{-}(F_{-}(W))=F_{-}(U)=W$. Therefore $F^{2\cdotp n}_{-}(W)=W$ for each $n\in N$.

Then $\forall n\in N\colon F^{2\cdotp n+1}_{-}(U)=F_{-}(F^{2\cdotp n}_{-}(U))=F_{-}(U)=W$ and\newline$\forall n\in N\colon F^{2\cdotp n+1}_{-}(W)=F_{-}(F^{2\cdotp n}_{-}(W))=F_{-}(W)=U$.

If $V=U\dot{\cup}W$ and for every $n\in N$ the equalities $F^{2\cdotp n+1}_{-}(U)=W$ and\newline$F^{2\cdotp n+1}_{-}(W)=U$ hold then $V=U\dot{\cup}W$ and $F_{-}(U)=W$ and $F_{-}(W)=U$ i.e. $F$ is bipartite.

\end{proof}

\begin{lemma}
Let $F\colon V\leadsto V$ be a strict and undirected multifunction, $U,W\in P(V)-\{\emptyset\}$ and $n\in N$. If $F$ is $(U,W)-$bipartite then $\forall(u,w)\in U\times W\colon u\notin F^{(2\cdotp n)\cup}(w)$ and $\forall u_{1},u_{2}\in U\colon u_{1}\notin F^{(2\cdotp n+1)\cup}(u_{2})$ and\newline$\forall w_{1},w_{2}\in W\colon w_{1}\notin F^{(2\cdotp n+1)\cup}(w_{2})$.

\end{lemma}

\begin{proof}
Assume that $F$ is $(U,W)-$bipartite. Then $\emptyset=U\cap W=F^{2\cdotp n}_{-}(U)\cap W=F^{(2\cdotp n)\cup}_{-}(U)\cap W$; or equivalently, $\forall(u,w)\in U\times W\colon u\notin F^{(2\cdotp n)\cup}(w)$. Notice that $\emptyset=U\cap W=U\cap F^{2\cdotp n+1}_{-}(U)=U\cap F^{(2\cdotp n+1)\cup}_{-}(U)$ iff $\forall u_{1},u_{2}\in U\colon u_{1}\notin F^{(2\cdotp n+1)\cup}(u_{2})$ and so $\emptyset=U\cap W=W\cap F^{2\cdotp n+1}_{-}(W)=W\cap F^{(2\cdotp n+1)\cup}_{-}(W)$ i.e. $\forall w_{1},w_{2}\in W\colon w_{1}\notin F^{(2\cdotp n+1)\cup}(w_{2})$. 
\end{proof}

\section{Connectedness}

Although there is a definition of connectedness based on the concept of path\cite{bollobas}, we will use the definition based on walks.
\begin{definition}\cite{jung}
Let $G=(V,E)$ be a simple graph and $u,w\in V$. Then $\{u,w\}$ is connected in $G$ iff $Walk_{u-w}(G\mid 1)\neq\emptyset$. We say that $G$ is connected iff all pairs of vertices of $G$ are connected in $G$.
\end{definition}
After translating this into multifunctions we get the following definition.
\begin{definition}
Let $F\colon V\leadsto V$ be a multifunction and $u,w\in V$. Then $(u,w)$ is connected in $F$ iff $Walk_{u\to w}(F,V\mid 1)\neq\emptyset$. We say that $F$ is connected iff\newline$\forall u,w\in V\colon Walk_{u\to w}(F,V\mid 1)\neq\emptyset$.

\end{definition}
We verify that this is good.

\begin{lemma}
Let $F\colon V\leadsto V$ be a simple graph multifunction and $u,w\in V$. Then $(u,w)$ is connected in $F$ iff $\{u,w\}$ is connected in $\tilde{G_{F}}=(V,\tilde{E_{F}})$ and $F$ is connected iff $\tilde{G_{F}}=(V,\tilde{E_{F}})$ is connected.

\end{lemma}

\begin{proof}
$(u,w)$ is connected in $F$ iff $\emptyset\neq Walk_{u\to w}(F,V\mid 1)=$\newline$=\bigcup_{n\in N}Walk_{u\to w}(F,V,n)$; or equivalently, there exists $n\in N$ such that $\emptyset\neq Walk_{u\to w}(F,V,n)=Walk_{u-w}((V,\tilde{E_{F}}),n)$. We may say equivalently that $\emptyset\neq\bigcup_{n\in N}Walk_{u-w}((V,\tilde{E_{F}}),n)=Walk_{u-w}((V,\tilde{E_{F}})\mid 1)=Walk_{u\to w}(\tilde{G_{F}}\mid 1)$ i.e. $\{u,w\}$ is connected in $\tilde{G_{F}}=(V,\tilde{E_{F}})$.

$F$ is connected iff $\forall u,w\in V\colon\emptyset\neq Walk_{u\to w}(F,V\mid 1)$; or equivalently, $\forall u,w\in V\colon\{u,w\}$ is connected in $\tilde{G_{F}}=(V,\tilde{E_{F}})$ i.e. $\tilde{G_{F}}=(V,\tilde{E_{F}})$ is connected.
\end{proof}

We operate everywhere with iterations and therefore we will describe connectedness by $F^{\infty}_{\mid m}$.

\begin{lemma}
Let $F,G\colon V\leadsto V$ be a simple graph multifunction, $m,k\in N$ and $u,w\in V$. Then$\colon$
\begin{enumerate}[(i)]
\item$(u,w)$ is connected in $F$ iff $u\in F^{\infty\cup}_{\mid 1}(w)$
\item$F$ is connected iff $V\in Clique(F^{\infty\cup}_{\mid 1},V)$
\item if $F$ is connected then $G$ is connected, assuming that $F\subset G$
\item if $F$ is connected then $F$ is strict
\item$(u,w)$ is connected in $F^{m\cup}$ iff $u\in F^{\infty\cup}_{\mid m}(w)$ 
\item$F^{m\cup}$ is connected iff $V\in Clique(F^{\infty}_{\mid m},V)$
\item if $F^{m\cup}$ is connected then $G^{m\cup}$ is connected, assuming that $F\subset G$
\item if $F^{m\cup}$ is connected then $F^{k\cup}$ is connected, assuming that $k\mid m$
\item if exists $n\in N$ such that $F^{n\cup}$ is connected then $F$ is connected
\item if $F$ is disconnected iff $F^{n\cup}$ is disconnected for each $n\in N$
\item if $F^{2\cup}$ is connected then $F$ is connected and not bipartite
\end{enumerate}

\end{lemma}

\begin{proof}
Fix $u,w\in V$.

\begin{enumerate}[(i)]
\item$(u,w)$ is connected in $F$ iff $\emptyset\neq Walk_{u\to w}(F,V\mid 1)=$\newline$=\bigcup_{n\in N}Walk_{u\to w}(F,V,n)$; or equivalently, there exists $n\in N$ such that $Walk_{u\to w}(F,V,n)\neq\emptyset$. We may say equivalently that there exists $n\in N$ such that $u\in F^{n\cup}(w)$ i.e. $u\in\bigcup_{n\in N}F^{n\cup}(w)=F^{\infty\cup}_{\mid 1}(w)$.
\item$F$ is connected iff $\forall u,w\in V\colon u\in F^{\infty\cup}_{\mid 1}(w)$ i.e. $V\in Clique(F^{\infty\cup}_{\mid 1},V)$.
\item If $F\subset G$ then $F^{\infty\cup}_{\mid 1}=F^{+\infty\cup}_{\mid 1}\subset G^{+\infty\cup}_{\mid 1}=G^{\infty\cup}_{\mid 1}$ and so $Clique(F^{\infty\cup}_{\mid 1},V)\subset Clique(G^{\infty\cup}_{\mid 1},V)$. If $V\in Clique(F^{\infty\cup}_{\mid 1},V)$ then $V\in Clique(G^{\infty\cup}_{\mid 1},V)$; or equivalently, if $F$ is connected then $G$ is connected.  
\item$F$ is connected iff $\text{const}^{V}=F^{\infty\cup}_{\mid 1}$. Then $(F^{\infty\cup}_{\mid 1})_{+}(\emptyset)=(F^{+\infty\cup}_{\mid 1})_{+}(\emptyset)=(\text{const}^{V})_{+}(\emptyset)=\emptyset$ i.e. $F^{\infty\cup}_{\mid 1}$ is strict; or equivalently, $F$ is strict.
\item$(u,w)$ is connected in $F^{m\cup}$ iff $u\in(F^{m\cup})^{\infty\cup}_{\mid 1}(w)=F^{\infty\cup}_{\mid m}(w)$
\item$F^{m\cup}$ is connected iff $\forall u,w\in V\colon u\in F^{\infty\cup}_{\mid m}(w)$; or equivalently,\newline$V\in Clique(F^{\infty\cup}_{\mid m},V)$.
\item If $F\subset G$ then $F^{\infty\cup}_{\mid m}=F^{+\infty\cup}_{\mid m}\subset G^{+\infty\cup}_{\mid m}=G^{\infty\cup}_{\mid m}$ and so\newline$Clique(F^{\infty\cup}_{\mid m},V)\subset Clique(G^{\infty\cup}_{\mid m},V)$. If $V\in Clique(F^{\infty\cup}_{\mid m},V)$ then $V\in Clique(G^{\infty\cup}_{\mid m},V)$; or equivalently, if $F^{m\cup}$ is connected then $G^{m\cup}$ is connected.  
\item If $k\mid m$ then $F^{\infty\cup}_{\mid m}\subset F^{\infty\cup}_{\mid k}$ and so $Clique(F^{\infty\cup}_{\mid m},V)\subset Clique(F^{\infty\cup}_{\mid k},V)$. If $V\in Clique(F^{\infty\cup}_{\mid m},V)$ then $V\in Clique(F^{\infty\cup}_{\mid k},V)$; or equivalently, if $F^{m\cup}$ is connected then $F^{k\cup}$ is connected.
\item Assume that there exists $n\in N$ such that $F^{n\cup}$ is connected. Then $F=F^{1\cup}$ is connected because $1\mid n$.
\item If $F^{n\cup}$ is disconnected for each $n\in N$ then $F=F^{1\cup}$ is disconnected. But if $F$ is disconnected then $V\notin Clique(F^{\infty\cup}_{\mid 1},V)$. Notice that $F^{\infty\cup}_{\mid n}\subset F^{\infty\cup}_{\mid 1}$ for each $n\in N$. Then $Clique(F^{\infty\cup}_{\mid n},V)\subset Clique(F^{\infty\cup}_{\mid 1},V)$ for each $n\in N$. Thus $V\notin Clique(F^{\infty\cup}_{\mid n},V)$ for each $n\in N$ i.e. $F^{n\cup}$ is disconnected for each $n\in N$.
\item Assume that $F^{2\cup}$ is connected. Then $F$ is connected and strict. Then also $\forall u,w\in V\colon u\in F^{\infty\cup}_{\mid 2}(w)=\bigcup_{n\in N}F^{(2\cdotp n)\cup}(w)$. Assume that there exist $U,W\in P(V)-\{\emptyset\}$ such that $F$ is $(U,W)-$bipartite i.e. $F$ is bipartite. Let $u\in U$ and $w\in W$. Then $u\in\bigcup_{n\in N}F^{(2\cdotp n)\cup}(w)$ i.e. $\exists n\in N\colon u\in F^{(2\cdotp n)\cup}(w)$. But for every strict and bipartite simple graph multifunction $\forall n\in N\colon u\notin F^{(2\cdotp n)\cup}(w)$, contradiction. 
\end{enumerate}

\end{proof}

Here we prove the K\"onig theorem for multifunctions.

\begin{theorem}
Let $F\colon V\leadsto V$ be a connected simple graph multifunction. Then $F$ is bipartite iff for every $n\in N$ and every $v\in V$ the formula $v\notin F^{(2\cdotp n+1)\cup}(v)$ is fulfilled.

\end{theorem}

\begin{proof}
Assume that $F$ is bipartite i.e. there exist $U,W\in P(V)-\{\emptyset\}$ such that $F$ is $(U,W)-$bipartite. Without loss of generality we assume that there exist $u\in U$ and $n\in N$ such that $u\in F^{(2\cdotp n+1)\cup}(u)$. But $\forall n\in N\colon\forall u\in U\colon u\notin F^{(2\cdotp n+1)\cup}(u)$, contradiction. 

Assume that $\forall v\in V\colon\forall n\in N\colon v\notin F^{(2\cdotp n+1)\cup}(v)$. Fix $v\in V$. Denote $U=F^{\infty\cup}_{\mid 2}(v)$ and $W\colon=F^{\infty\cup}_{\neg\mid 2}(v)=(F^{\infty\cup}_{\mid 1}-F^{\infty\cup}_{\mid 2})(v)=F^{\infty\cup}_{\mid 1}(v)-F^{\infty\cup}_{\mid 2}(v)$. $V=\text{const}^{V}(v)=F^{\infty\cup}_{\mid 1}(v)=U\dot{\cup}W$ because $F$ is connected. Assume that $\emptyset\neq U\cap F_{-}(U)$. Notice that $\emptyset\neq F^{\infty\cup}_{\mid 2}(v)\cap F_{-}(F^{\infty\cup}_{\mid 2}(v))$ iff there exist $u,w\in V$ such that $u\in F(w),w\in F^{\infty\cup}_{\mid 2}(v)=U$ and $v\in F^{\infty\cup}_{\mid 2}(u)$. We may say equivalently that there exist $n,m\in N$ and $u,w\in V$ such that $u\in F(w),w\in F^{(2\cdotp n)\cup}(v)$ and $v\in F^{(2\cdotp m)\cup}(u)$. Then there exist $u,w\in V$ and $k\in N,k=m+n$ such that $u\in F(w)$ and $w\in F^{(2\cdotp n+2\cdotp m)\cup}(u)=F^{(2\cdotp(n+m))\cup}(u)=F^{(2\cdotp k)\cup}(u)$. Then there exist $u\in V$ and $k\in N$ such that $u\in F^{(2\cdotp k+1)\cup}(u)$, contradiction. We do the same for the set $W$. Indeed, notice that $\emptyset\neq F^{\infty\cup}_{\neg\mid 2}(v)\cap F_{-}(F^{\infty\cup}_{\neg\mid 2}(v))$ iff there exist $u,w\in V$ such that $u\in F(w),w\in F^{\infty\cup}_{\neg\mid 2}(v)$ and $v\in F^{\infty\cup}_{\neg\mid 2}(u)$. We may say equivalently that there exist $n,m\in N$ and $u,w\in V$ such that $u\in F(w),w\in F^{(2\cdotp n+1)\cup}(v)$ and $v\in F^{(2\cdotp m+1)\cup}(u)$. Then there exist $u,w\in V$ and $k\in N,k=m+n+1$ such that $u\in F(w)$ and $w\in F^{(2\cdotp n+1+2\cdotp m+1)\cup}(u)=F^{(2\cdotp(n+m+1))\cup}(u)=F^{(2\cdotp k)\cup}(u)$. Then there exist $u\in V$ and $k\in N$ such that $u\in F^{(2\cdotp k+1)\cup}(u)$, contradiction.
\end{proof}

Here we go to the main result. We need a simple fact from the theory of numbers.
\begin{lemma}
If $a,b,c\in N$ then either $2\mid(a+b+2\cdotp c+1)$ or $2\mid(a+b+2\cdotp(2\cdotp c+1))$.
\end{lemma}

\begin{proof}
Fix $a,b,c\in N$. Then $2\mid(a+b+2\cdotp c+1)$ iff there exists $q\in Z$ such that $a+b+1+2\cdotp c=2\cdotp q$; or equivalently, there exists $q\in Z$ such that $2\cdotp(q-c)=a+b+1$ i.e. $a+b+1$ is even. Notice that $2\mid(a+b+2\cdotp(2\cdotp c+1))$ iff there exists $q\in Z$ such that
$a+b+2\cdotp(2\cdotp c+1)=2\cdotp q$; or equivalently, there exists $q\in Z$ such that $2\cdotp(q-2\cdotp c-1)=a+b$ i.e. $a+b$ is even.
\end{proof}

\begin{lemma}
Let $F\colon V\leadsto V$ be a connected simple graph multifunction. If $F$ is not bipartite then $F^{2\cup}$ is connected.
\end{lemma}

\begin{proof}

Assume that $F$ is connected and $F$ is not bipartite. Fix $u,w\in V,u\neq w$. Then $u\in F^{\infty\cup}_{\mid 1}(w)=\bigcup_{n\in N}F^{n\cup}(w)$ i.e. there exists $n\in N$ such that $u\in F^{n\cup}(w)$. But either $2\mid n$ or not $2\mid n$. If $2\mid n$ then $F^{n\cup}\subset F^{\infty\cup}_{\mid 2}$ and therefore $u\in F^{\infty\cup}_{\mid 2}(w)$. Suppose that not $2\mid n$. Here we use the assumption that $F$ is not bipartite. Then there exist $v\in V$ and $c\in N$ such that $v\in F^{(2\cdotp c+1)\cup}(v)$. For every $c\in N,v\in V$ the formula $v\in F^{(2\cdotp(2\cdotp c+1))\cup}(v)$ holds. Then there exist $v\in V$ and $c\in N$ such that $v\in F^{(2\cdotp c+1)\cup}(v)$ and $v\in F^{(2\cdotp(2\cdotp c+1))\cup}(v)$. But $F$ is connected and so $u\in F^{\infty\cup}_{\mid 1}(v)$ and $v\in F^{\infty\cup}_{\mid 1}(w)$ i.e. there exist $a,b\in N$ such that $u\in F^{a\cup}(v)$ and $v\in F^{b\cup}(w)$. Then $u\in F^{(a+b+2\cdotp c+1)\cup}(w)$ and $u\in F^{(a+b+2(2\cdotp c+1))\cup}(w)$. But either $2\mid(a+b+2\cdotp c+1)$ or $2\mid(a+b+2\cdotp(2\cdotp c+1))$. If $2\mid(a+b+2\cdotp c+1)$ then $F^{(a+b+2\cdotp c+1)\cup}\subset F^{\infty\cup}_{\mid 2}$ and if $2\mid(a+b+2\cdotp(2\cdotp c+1))$ then $F^{(a+b+2\cdotp(2\cdotp c+1))\cup}\subset F^{\infty\cup}_{\mid 2}$. Therefore either $u\in F^{(a+b+2\cdotp c+1)\cup}(w)\subset F^{\infty\cup}_{\mid 2}(w)$ or $u\in F^{(a+b+2(2\cdotp c+1))\cup}(w)\subset F^{\infty\cup}_{\mid 2}(w)$. Then $u\in F^{\infty\cup}_{\mid 2}(w)$. $F^{2\cup}$ is connected because $u,w\in V$ are arbitrary.

\end{proof}

Taking it all together we have the fundamental theorem.
\begin{theorem}
Let $F\colon V\leadsto V$ be a connected simple graph multifunction. Then $F$ is not bipartite iff $F^{2\cup}$ is connected. Furthermore $F$ is bipartite iff for every $n\in N$ the multifunction $F^{(2\cdotp n)\cup}$ is disconnected.
\end{theorem} 

\begin{proof}
For every connected simple graph multifunction $F\colon V\leadsto V$ if $F$ is not bipartite then $F^{2\cup}$ is connected. For every simple graph multifunction $F\colon V\leadsto V$ if $F^{2\cup}$ is connected then $F$ is connected and not bipartite. Thus we have equivalence $F$ is bipartite iff $F^{2\cup}$ is disconnected, assuming that $F$ is connected. But $F^{(2\cdotp n)\cup}=(F^{2\cup})^{n\cup}$ for each $n\in N$. Notice that $F$ is disconnected iff $F^{n\cup}$ is disconnected for each $n\in N$. Substituting $F^{2\cup}$ for $F$ we have the last equivalence. 
\end{proof}

\section{Filters and ideals of multifunctions}

\subsection{Families of subsets}
We considered the families of subsets of the set of vertices $V$. So far we have been working on $Ind(F,V)$. This family is neither a filter nor an ideal. 
\begin{lemma}
Let $F\colon V\leadsto V$ be a multifunction and $U,W\in P(V)$. Then$\colon$
\begin{enumerate}[(i)]
\item$\emptyset\in Ind(F,V)$
\item$V\in Ind(F,V)$ iff $F$ is trivial
\item if $U\subset W$ and $W\in Ind(F,V)$ then $U\in Ind(F,V)$
\item if $U\in Ind(F,V)$ then $U\cap W\in Ind(F,V)$
 
\end{enumerate}

\end{lemma}

\begin{proof}
It is easily seen that$\colon$
\begin{enumerate}[(i)]
\item$\emptyset=\emptyset\cap F_{-}(\emptyset)$ iff $\emptyset\in Ind(F,V)$.
\item$V\in Ind(F,V)$ iff $V\cap F_{-}(V)=\emptyset$; or equivalently, $F_{-}(V)=\emptyset$ i.e. $F$ is trivial.
\item Assume that $U\subset W$ and $W\in Ind(F,V)$. Then $U\subset W$ and $W\cap F_{-}(W)=\emptyset$. But $F_{-}(U)\subset F_{-}(W)$ and so $U\cap F_{-}(U)\subset W\cap F_{-}(U)\subset W\cap F_{-}(W)$ i.e. $U\in Ind(F,V)$.
\item If $U\in Ind(F,V)$ then $\emptyset=\emptyset\cap W\cap F_{-}(W)=U\cap F_{-}(U)\cap W\cap F_{-}(W)=(U\cap W)\cap(F_{-}(U)\cap F_{-}(W))\supset(U\cap W)\cap F_{-}(U\cap W)$ and so $\emptyset=(U\cap W)\cap F_{-}(U\cap W)$ i.e. $U\cap W\in Ind(F,V)$. 

\end{enumerate}
\end{proof}
However there are filters and ideals\cite{comfort} in graph theory.

\begin{definition}
Let $F\colon V\leadsto V$ be a multifunction, $A\in P(V)$ and $\kappa$ be a cardinal. Then$\colon$

\begin{itemize}
\item$Neigh_{A}(F,V)=\{U\in P(V)\mid F_{-}(U)\subset A\}$ is the family of sets with neighbors from $A$ 
\item$Wall_{A}(F,V)=\{U\in P(V)\mid A\subset F_{+}(U)\}$ is the family of sets that are like a wall for $A$  
\item$Neigh_{\kappa}(F,V)=\{U\in P(V)\mid\sharp F_{-}(U)<\kappa\}\subset P(V)$ is the family of sets with number of neighbors less then $\kappa$
\item$Wall_{\kappa}(F,V)=\{U\in P(V)\mid\sharp F_{+}(U)^{c}<\kappa\}$ is the family of sets that are complements of above
\item$Isol(F,V)=Neigh_{\emptyset}(F,V)$ is the set of isolated sets
\item$Build(F,V)=Wall_{V}(F,V)$ is the set of nonisolated sets i.e. sets forming some structures
 
\end{itemize}

\end{definition}
If we substitute $F=\{\cdotp\}$ then $Neigh_{A}(F,V)=P(A),\hspace{1pt}Wall_{A}(F,V)=\{U\in P(V)\mid A\subset U\},\hspace{1pt}Neigh_{\kappa}(F,V)=\{U\in P(V)\mid\sharp U<\kappa\},\hspace{1pt}Wall_{\kappa}(F,V)=\{U\in P(V)\mid\sharp U^{c}<\kappa\}$. These are all principal, cofinite and finite ideals and filters over sets.  
\begin{lemma}
Let $F\colon V\leadsto V$ be a multifunction, $A\in P(V)$ and $\kappa$ be a cardinal. Then $Neigh_{\kappa}(F,V)$ and $Neigh_{A}(F,V)$ are ideals.

\end{lemma}

\begin{proof}
Fix $U,W\in P(V)$. Notice that if $U\subset W$ then $F_{-}(U)\subset F_{-}(W)$. 

If $U\subset W$ and $W\in Neigh_{A}(F,V)$ then $F_{-}(U)\subset F_{-}(W)\subset A$ and so $U\in Neigh_{A}(F,V)$. If $U\in Neigh_{A}(F,V)$ and $W\in Neigh_{A}(F,V)$ i.e. $F_{-}(U)\subset A$ and $F_{-}(W)\subset A$ then $F_{-}(U\cup W)=F_{-}(U)\cup F_{-}(W)\subset A$ i.e. $U\cup W\in Neigh_{A}(F,V)$. 

If $U\subset W$ and $W\in Neigh_{\kappa}(F,V)$ then $\sharp F_{-}(U)\le\sharp F_{-}(W)<\kappa$ and so $U\in Neigh_{\kappa}(F,V)$. If $U\in Neigh_{\kappa}(F,V)$ and $W\in Neigh_{\kappa}(F,V)$ i.e. $\sharp F_{-}(U)<\kappa$ and $\sharp F_{-}(W)<\kappa$ then $\sharp F_{-}(U\cup W)=\sharp(F_{-}(U)\cup F_{-}(W))\le\sharp F_{-}(U)+\sharp F_{-}(W)<\kappa$ i.e. $U\cup W\in Neigh_{\kappa}(F,V)$.

\end{proof}

In the case of undirected multifunction $Isol(F,V)$ is indeed the family of sets that make the multifunction nonstrict.
\begin{lemma}
Let $F\colon V\leadsto V$ be an undirected multifunction and $U\in P(V)$. Then$\colon$
\begin{enumerate}[(i)]
\item$\emptyset\in Isol(F,V)$
\item$U\in Isol(F,V)$ iff $F|_{U}=\text{const}^{\emptyset}$
\item$V\in Isol(F,V)$ iff $F$ is trivial
\item$Isol(F,V)=\{\emptyset\}$ iff $F$ is strict
\end{enumerate}

\end{lemma}

\begin{proof}
It is easily seen that$\colon$
\begin{enumerate}[(i)]
\item$\emptyset\in Isol(F,V)$ iff $F_{-}(\emptyset)=\emptyset$.
\item$U\in Isol(F,V)$ iff $\emptyset=F_{-}(U)$; or equivalently, $\forall(v,u)\in V\times U\colon u\notin F(v)=F^{-1}(v)$ because $F$ is undirected. We may say equivalently that $\forall u\in U\colon\forall v\in V\colon v\notin F(u)$; or equivalently, $\forall u\in U\colon F(u)=\emptyset$ i.e. $F|_{U}=\text{const}^{\emptyset}$.
\item$V\in Isol(F,V)$ iff $F=F|_{V}=\text{const}^{\emptyset}$ i.e. $F$ is trivial.
\item$F$ is nonstrict iff $\exists U\in P(V)-\{\emptyset\}\colon F|_{U}=\text{const}^{\emptyset}$; or equivalently, $\exists U\in P(V)-\{\emptyset\}\colon U\in Isol(F,V)$ i.e. $Isol(F,V)\neq\{\emptyset\}$.

\end{enumerate}
\end{proof}
We use only strict multifunctions and therefore we will not consider $Isol(F,V)$. We know that there exists the bijective correspondence between ideals and filters\cite{comfort}. This is the function $\uparrow^{d}\colon P(P(V))-\{\emptyset\}\to P(P(V));\Phi\mapsto\Phi^{d}=\{A\in P(V)\mid A^{c}\in\Phi\}$. It is easily seen that $\uparrow^{d}\circ\uparrow^{d}=\text{id}_{P(P(V))}$ and for every $\emptyset\neq\Phi\in P(P(V))$ the set $\Phi$ is ideal iff $\Phi^{d}$ is filter. 
\begin{lemma}
Let $F\colon V\leadsto V$ be a multifunction, $A\in P(V)$ and $\kappa$ be a cardinal. Then$\colon$

\begin{enumerate}[(i)]
\item$Neigh_{A}(F,V)^{d}=Wall_{A^{c}}(F,V)$
\item$Neigh_{\kappa}(F,V)^{d}=Wall_{\kappa}(F,V)$
\item$Isol(F,V)^{d}=Build(F,V)$
\item$Build(F,V)=\{V\}$ iff $F$ is strict
\end{enumerate}

\end{lemma}

\begin{proof}
Fix $U\in P(V)$.
\begin{enumerate}[(i)]
\item$U\in Neigh_{A}(F,V)^{d}$ iff $U^{c}\in Neigh_{A}(F,V)$ i.e. $A\supset F_{-}(U^{c})=(\uparrow^{c}\circ F_{+}\circ\uparrow^{c})(\uparrow^{c}(U))=(\uparrow^{c}\circ F_{+}\circ\uparrow^{c}\circ\uparrow^{c})(U)=(\uparrow^{c}\circ F_{+})(U)=F_{+}(U)^{c}$. We may say equivalently that $F_{+}(U)=F_{+}(U)^{cc}\supset A^{c}$ i.e. $U\in Wall_{A^{c}}(F,V)$.
\item$U\in Neigh_{\kappa}(F,V)^{d}$ iff $U^{c}\in Neigh_{\kappa}(F,V)$ i.e. $\kappa>\sharp F_{-}(U^{c})=\sharp F_{-}(V-U)=\sharp(V-F_{+}(U))=\sharp F_{+}(U)^{c}$; or equivalently,\newline$U\in Wall_{\kappa}(F,V)$.
\item$Isol(F,V)^{d}=Neigh_{\emptyset}(F,V)^{d}=Wall_{\emptyset^{c}}(F,V)=Wall_{V}(F,V)=$\newline$=Build(F,V)$.
\item$F$ is strict iff $Isol(F,V)=\{\emptyset\}$; or equivalently, $Build(F,V)=$\newline$=Isol(F,V)^{d}=\{\emptyset\}^{d}=\{V\}$.
\end{enumerate}
\end{proof}

In this way we got the ideals and filters on the set of vertices that seem to be like set theoretic ideals and filters. So the filter $Wall_{\kappa}(F,V)$ behaves like $\kappa-$finite filter. 
\begin{lemma}
Let $F\colon V\leadsto V$ be a strict multifunction such that $\sharp V\ge\kappa\ge\aleph_{0}$. Then $Wall_{\kappa}(F,V)$ is the proper filter.
\end{lemma}

\begin{proof}
Assume that $F$ is strict. Then $\emptyset\in Wall_{\kappa}(F,V)$ iff $\sharp F_{+}(\emptyset)^{c}=\sharp\emptyset^{c}=\sharp V<\kappa$ but $\kappa\le\sharp V$ and contradiction.
\end{proof}

Consider the $Wall_{A}(F,V)$ as a principal filter in sense of multifunctions. We define the filter generated by a family of sets. Fix $\Gamma\in P(P(V))$ and $F\colon V\leadsto V$. Let $F_{+}^{-1}(\Gamma)=\{U\in P(V)\mid F_{+}(U)\in\Gamma\}$ be the preimage of small preimage $F_{+}\colon P(V)\to P(V);U\mapsto F_{+}(U)=\{v\in V\mid F(v)\subset U\}$.  

\begin{definition}
Let $F\colon V\leadsto V$ be a multifunction and $\Gamma\in P(P(V))$. Then$\colon$

\begin{enumerate}[(i)]
\item$\alpha_{\Gamma}(F,V)=\{\Phi\in\alpha(V)\mid F_{+}^{-1}(\Gamma)\subset\Phi\}\subset\alpha(V)$ is set of filters containing the family $F_{+}^{-1}(\Gamma)$
\item$\delta^{+}(F,V)=\{\Phi\in P(P(V))-\{\emptyset\}\mid\forall U,W\in P(V)\colon(U\in\Phi\land W\in\Phi\Rightarrow U\cap W\in\Phi)\land(U\in\Phi\land U\subset F_{+}(W)\Rightarrow W\in\Phi)\}\subset P(P(V))$ is set of $F_{+}-$filters 
\item$\delta_{\Gamma}^{+}(F,V)=\{\Phi\in\delta^{+}(F,V)\mid\Gamma\subset\Phi\}\subset\delta^{+}(F,V)$ is set of $F_{+}-$filters containing the family $\Gamma$

\end{enumerate}
\end{definition}

\begin{lemma}
Let $F\colon V\leadsto V$ be a multifunction, $U,W\in P(V)$, $\Gamma,\Gamma_{1},\Gamma_{2}\in P(P(V))$ and $\Psi\in\delta^{+}(F,V),\rho\subset\delta^{+}(F,V)$. Then$\colon$

\begin{enumerate}[(i)]
\item$P(V)\in\alpha_{\Gamma}(F,V)$ and $P(V)\in\delta_{\Gamma}^{+}(F,V)$
\item if $\Gamma_{1}\subset\Gamma_{2}$ then $\alpha_{\Gamma_{2}}(F,V)\subset\alpha_{\Gamma_{1}}(F,V)$ and $\delta^{+}_{\Gamma_{2}}(F,V)\subset\delta^{+}_{\Gamma_{1}}(F,V)$
\item$\alpha_{P(V)}(F,V)=\{P(V)\}=\delta^{+}_{P(V)}(F,V)$
\item$\alpha_{\emptyset}(F,V)=\alpha(V)$ and $\delta^{+}_{\emptyset}(F,V)=\delta^{+}(F,V)$
\item if $F$ is strict then $\alpha_{\{\emptyset\}}(F,V)=\{P(V)\}$ 
\item$\emptyset\in\Psi\Leftrightarrow\Psi=P(V)$ 
\item$\delta^{+}_{\{\emptyset\}}(F,V)=\{P(V)\}$
\item if $F$ is trivial then $\alpha_{\Gamma}(F,V)=\begin{cases}\{P(V)\}$ iff $V\in\Gamma\\\alpha(V)$ iff $V\notin\Gamma\end{cases}$
\item if $F$ is trivial then $\delta^{+}(F,V)=\{P(V)\}=\delta^{+}_{\Gamma}(F,V)$  
\item$\bigcap\alpha_{\Gamma}(F,V)\in\alpha_{\Gamma}(F,V)$ 
\item$\bigcap\rho\in\delta^{+}(F,V)$
\item$\bigcap\delta^{+}_{\Gamma}(F,V)\in\delta^{+}_{\Gamma}(F,V)$
\end{enumerate}

\end{lemma}

\begin{proof}
Fix $\Phi\in\alpha(V),\Psi\in\delta^{+}(F,V)$.
\begin{enumerate}[(i)]
\item$P(V)\in\alpha_{\Gamma}(F,V)$ because $P(V)\in\alpha(V)$ and $F_{+}^{-1}(\Gamma)\subset P(V)$ and $P(V)\in\delta^{+}_{\Gamma}(F,V)$ because $P(V)\in\delta^{+}(F,V)$ and $\Gamma\subset P(V)$
\item If $\Gamma_{1}\subset\Gamma_{2}$ then $F_{+}^{-1}(\Gamma_{1})\subset F_{+}^{-1}(\Gamma_{2})$. Assume that $\Phi\in\alpha_{\Gamma_{2}}(F,V)$. Then $F_{+}^{-1}(\Gamma_{2})\subset\Phi$ and so $F_{+}^{-1}(\Gamma_{1})\subset\Phi$ i.e. $\Phi\in\alpha_{\Gamma_{1}}(F,V)$. Assume that $\Psi\in\delta^{+}_{\Gamma_{2}}(F,V)$. Then $\Gamma_{2}\subset\Psi$ and so $\Gamma_{1}\subset\Psi$ i.e. $\Psi\in\delta^{+}_{\Gamma_{1}}(F,V)$.
\item Notice that $F_{+}^{-1}(P(V))=P(V)$ and so $\alpha_{P(V)}(F,V)=\{\Phi\in\alpha(V)\mid F_{+}^{-1}(P(V))\subset\Phi\}=\{\Phi\in\alpha(V)\mid P(V)\subset\Phi\}=\{P(V)\}$ and\newline$\delta^{+}_{P(V)}(F,V)=\{\Psi\in\delta^{+}(F,V)\mid P(V)\subset\Psi\}=\{P(V)\}$
\item$\alpha_{\emptyset}(F,V)=\{\Phi\in\alpha(V)\mid F_{+}^{-1}(\emptyset)\subset\Phi\}=\{\Phi\in\alpha(V)\mid\emptyset\subset\Phi\}=\alpha(V)$ and $\delta^{+}_{\emptyset}(F,V)=\{\Psi\in\delta^{+}(F,V)\mid\emptyset\subset\Psi\}=\delta^{+}(F,V)$
\item Assume that $F$ is strict i.e. $\emptyset\in F_{+}^{-1}(\{\emptyset\})$. Then $\alpha_{\{\emptyset\}}(F,V)=\{\Phi\in\alpha(V)\mid F_{+}^{-1}(\{\emptyset\})\subset\Phi\}\subset\{\Phi\in\alpha(V)\mid\emptyset\in\Phi\}=\{P(V)\}$ and so $\alpha_{\{\emptyset\}}(F,V)=\{P(V)\}$
\item If $\emptyset\in\Psi$ then for every $W\in P(V)$ such that $\emptyset\subset F_{+}(W)$ but this is true $W\in\Psi$ i.e. $\Psi=P(V)$. If $\Psi=P(V)$ then $\emptyset\in\Psi$.
\item$\delta^{+}_{\{\emptyset\}}(F,V)=\{\Psi\in\delta^{+}(F,V)\mid\{\emptyset\}\subset\Psi\}=\{P(V)\}$ because $\emptyset\in\Psi$ iff $\Psi=P(V)$
\item Notice that $(\text{const}^{\emptyset}_{+})^{-1}(\Gamma)=\begin{cases}P(V)$ iff $V\in\Gamma\\\emptyset$ iff $V\notin\Gamma\end{cases}$ and so $\alpha_{\Gamma}(\text{const}^{\emptyset},V)=\{\Phi\in\alpha(V)\mid(\text{const}^{\emptyset}_{+})^{-1}(\Gamma)\subset\Phi\}=\begin{cases}\{P(V)\}$ iff $V\in\Gamma\\\alpha(V)$ iff $V\notin\Gamma\end{cases}$
\item$\delta^{+}(\text{const}^{\emptyset},V)=\{P(V)\}$ because $\text{const}^{\emptyset}_{+}(W)=V$ and $\delta^{+}_{\Gamma}(\text{const}^{\emptyset},V)=\{\Psi\in\{P(V)\}\mid\Gamma\subset\Psi\}=\{P(V)\}$
\item$\alpha_{\Gamma}(F,V)\subset\alpha(V)$ and so $\bigcap\alpha_{\Gamma}(F,V)\in\alpha(V)$\cite{comfort}. If $Z\in F_{+}^{-1}(\Gamma)$ then for every $\Phi\in\alpha(V)$ such that $F_{+}^{-1}(\Gamma)\subset\Phi$ clearly $Z\in\Phi$ i.e. $F_{+}^{-1}(\Gamma)\subset\bigcap\alpha_{\Gamma}(F,V)$ and so $\bigcap\alpha_{\Gamma}(F,V)\in\alpha_{\Gamma}(F,V)$ is the smallest filter from $\alpha_{\Gamma}(F,V)$
\item If $U,W\in\bigcap\rho$ then $U\in\Psi$ and $W\in\Psi$ for each $\Psi\in\rho\subset\delta^{+}(F,V)$. Thus $U\cap W\in\Psi$ for each $\Psi\in\rho$ i.e. $U\cap W\in\bigcap\rho$. If $U\in\bigcap\rho$ and $U\subset F_{+}(W)$. Then $U\in\Psi$ and $U\subset F_{+}(W)$ for each $\Psi\in\rho\subset\delta^{+}(F,V)$. Thus $W\in\Psi$ for each $\Psi\in\rho$ i.e. $W\in\bigcap\rho$. Therefore $\bigcap\rho\in\delta^{+}(F,V)$.
\item$\delta^{+}_{\Gamma}(F,V)\subset\delta^{+}(F,V)$ and so $\bigcap\delta^{+}_{\Gamma}(F,V)\in\delta^{+}(F,V)$. If $Z\in\Gamma$ then for every $\Psi\in\delta^{+}(F,V)$ such that $\Gamma\subset\Psi$ clearly $Z\in\Psi$ i.e. $\Gamma\subset\bigcap\delta^{+}_{\Gamma}(F,V)$ and so $\bigcap\delta^{+}_{\Gamma}(F,V)\in\delta^{+}_{\Gamma}(F,V)$ is the smallest $F_{+}-$filter from $\delta^{+}_{\Gamma}(F,V)$. 
\end{enumerate}
\end{proof}

\begin{lemma}
Let $F\colon V\leadsto V$ be a multifunction, $A\in P(V)$ and $\Gamma\in P(P(V))$. Then$\colon$
\begin{enumerate}[(i)]
\item$Filtr_{\Gamma}(F,V)=\{U\in P(V)\mid\exists n\in N\colon\exists A_{1},\ldots,A_{n}\in\Gamma\colon\bigcap_{i=1}^{n}A_{i}\subset F_{+}(U)\}\in\alpha_{\Gamma}(F,V)$
\item$Filtr_{\{A\}}(F,V)=Wall_{A}(F,V)$
\item$Filtr_{\Gamma}(F,V)\subset\bigcap\delta^{+}_{\Gamma}(F,V)$

\end{enumerate}
\end{lemma}

\begin{proof}
Fix $A,U,W\in P(V),\Psi\in\delta^{+}(F,V)$. Notice that if $U\subset W$ then $F_{+}(U)\subset F_{+}(W)$.
\begin{enumerate}[(i)]
\item If $U\subset W$ and $U\in Filtr_{\Gamma}(F,V)$ then there exist $n\in N$ and\newline$A_{1},\ldots,A_{n}\in\Gamma$ such that $\bigcap_{i=1}^{n}A_{i}\subset F_{+}(U)\subset F_{+}(W)$ and so $W\in Filtr_{\Gamma}(F,V)$. Assume that $U\in Filtr_{\Gamma}(F,V)$ and $W\in Filtr_{\Gamma}(F,V)$ i.e. there exist $n,m\in N$ and $A_{1},\ldots,A_{n},B_{1},\ldots,B_{m}\in\Gamma$ such that $\bigcap_{i=1}^{n}A_{i}\subset F_{+}(U)$ and $\bigcap_{i=1}^{m}B_{i}\subset F_{+}(W)$. Then there exist $k\in N,k=m+n$ and $C_{1},\ldots,C_{k}\in\Gamma,C_{i}=\begin{cases}A_{i}$ iff $i\in\{1,\ldots,n\}\\B_{i-n}$ iff $i\in\{n+1,\ldots,n+m\}\end{cases}$ such that $\bigcap_{i=1}^{k}C_{i}=\bigcap_{i=1}^{n}A_{i}\cap\bigcap_{j=1}^{m}B_{j}\subset F_{+}(U)\cap F_{+}(W)=F_{+}(U\cap W)$ i.e. $U\cap W\in Filtr_{\Gamma}(F,V)$. Assume that $U\in F_{+}^{-1}(\Gamma)$. Then $F_{+}(U)\in\Gamma$ and so there exist $n\in N,n=1$ and $A\in\Gamma,A=F_{+}(U)$ such that $A\subset F_{+}(U)$ i.e. $U\in Filtr_{\Gamma}(F,V)$. Therefore $F_{+}^{-1}(\Gamma)\subset Filtr_{\Gamma}(F,V)$.
\item Thus $U\in Filtr_{\{A\}}(F,V)$ iff there exists $B\in\{A\}$ such that $B\subset F_{+}(U)$ i.e. $U\in Wall_{A}(F,V)$.   
\item If $U\in Filtr_{\Gamma}(F,V)$ then there exist $n\in N$ and $A_{1},\ldots,A_{n}\in\Gamma$ such that $\bigcap_{i=1}^{n}A_{i}\subset F_{+}(U)$. Fix $\Psi\in\delta^{+}_{\Gamma}(F,V)$. Then $A_{1},\ldots,A_{n}\in\Psi$ and $\bigcap_{i=1}^{n}A_{i}\in\Psi$ and so $U\in\Psi$ i.e. $U\in\bigcap\delta^{+}_{\Gamma}(F,V)$.
\end{enumerate}
\end{proof}

This is the reason we call $Wall_{A}(F,V)$ a $(F,V)-$principal filter.

We find out how to apply these filters to metric spaces\cite{bourbakitop}. Having multifunction we can build the metric space.

\begin{lemma}
Let $F\colon V\leadsto V$ be a simple graph multifunction. Then $d_{F}\colon V\times V\to R\cup\{\infty\}$ is metric on $V$ where $(u,w)\mapsto d_{F}(u,w)=$\newline$=\begin{cases}\min\{n\in N\mid u\in F^{n\cup}(w)\}$ iff $u\in F^{\infty\cup}_{\mid 1}(w)\\\infty$ iff $u\notin F^{\infty\cup}_{\mid 1}(w)\end{cases}$. If $F$ is connected then $\forall u,w\in V\colon d_{F}(u,w)\in N$.
\end{lemma}

\begin{proof}
Fix $u,v,w\in V$. If $d_{F}(u,w)=0$ then $u\in F^{0\cup}(w)=\{w\}$ i.e. $u=w$. But $d_{F}(v,v)\neq\infty$ because $F^{\infty\cup}_{\mid 1}$ is everywhereloop and $d_{F}(v,v)=\min\{n\in N\mid v\in F^{n\cup}(v)\}=0$ because $v\in\{v\}=F^{0\cup}(v)$. From the undirectedness of $F$ we conclude that $d_{F}(u,w)=\infty$ iff $u\notin F^{\infty\cup}_{\mid 1}(w)$; or equivalently, $w\notin F^{\infty\cup}_{\mid 1}(u)$ i.e. $d_{F}(w,u)=\infty$. We also conclude that $d_{F}(u,w)=\min\{n\in N\mid u\in F^{n\cup}(w)\}=\min\{n\in N\mid w\in F^{n\cup}(u)\}=d_{F}(w,u)$ iff $u\in F^{\infty\cup}_{\mid 1}(w)$ i.e. $w\in F^{\infty\cup}_{\mid 1}(u)$. Here we are proving the triangle inequality. At first notice that if $u\in F^{\infty\cup}_{\mid 1}(v)$ and $v\in F^{\infty\cup}_{\mid 1}(w)$ then $u\in F^{\infty\cup}_{\mid 1}(w)$ because $F^{\infty\cup}_{\mid 1}$ is undirected and transitive. Therefore the following possibility exist.
\begin{itemize}
\item$u\in F^{\infty\cup}_{\mid 1}(v),\hspace{1pt}v\in F^{\infty\cup}_{\mid 1}(w),\hspace{1pt}u\in F^{\infty\cup}_{\mid 1}(w)$. Assume that $d_{F}(u,v)=i$ and $d_{F}(v,w)=j$. Then $u\in F^{i\cup}(v)$ and $v\in F^{j\cup}(w)$. Then\newline$u\in\bigcup_{a\in F^{j\cup}(w)}F^{i\cup}(a)=F^{i\cup}_{\cup}(F^{j\cup}(w))=F^{i}_{\cup}(F^{j\cup}(w))=F^{(i+j)\cup}(w)$ and so $d_{F}(u,w)\le i+j=d_{F}(u,v)+d_{F}(v,w)$.
\item$u\notin F^{\infty\cup}_{\mid 1}(v),\hspace{1pt}v\notin F^{\infty\cup}_{\mid 1}(w),\hspace{1pt}u\notin F^{\infty\cup}_{\mid 1}(w)$. Then $d_{F}(u,w)=\infty\le\infty=\infty+\infty=d_{F}(u,v)+d_{F}(v,w)$.
\item$u\notin F^{\infty\cup}_{\mid 1}(v),\hspace{1pt}v\notin F^{\infty\cup}_{\mid 1}(w),\hspace{1pt}u\in F^{\infty\cup}_{\mid 1}(w)$. Then there exists $n\in N$ such that $d_{F}(u,w)=n$ and therefore $d_{F}(u,w)<\infty=\infty+\infty=d_{F}(u,v)+d_{F}(v,w)$.
\item$u\notin F^{\infty\cup}_{\mid 1}(v),\hspace{1pt}v\in F^{\infty\cup}_{\mid 1}(w),\hspace{1pt}u\notin F^{\infty\cup}_{\mid 1}(w)$. Then there exists $n\in N$ such that $d_{F}(v,w)=n$ and therefore $d_{F}(u,w)=\infty\le\infty+n=d_{F}(u,v)+d_{F}(v,w)$.
\item$u\in F^{\infty\cup}_{\mid 1}(v),\hspace{1pt}v\notin F^{\infty\cup}_{\mid 1}(w),\hspace{1pt}u\notin F^{\infty\cup}_{\mid 1}(w)$. Then there exists $n\in N$ such that $d_{F}(u,v)=n$ and therefore $d_{F}(u,w)=\infty\le n+\infty=d_{F}(u,v)+d_{F}(v,w)$.
  
\end{itemize}  
Taking it all together we get $d_{F}(u,w)\le d_{F}(u,v)+d_{F}(v,u)$.

Assume that $F$ is connected and $\exists u,w\in V\colon d_{F}(u,w)=\infty$. Then $F^{\infty\cup}_{\mid 1}=\text{const}^{V}$ i.e. $\forall u,w\in V\colon u\in F^{\infty\cup}_{\mid 1}(w)$ and $\exists u,w\in V\colon u\notin F^{\infty\cup}_{\mid 1}(w)$, contradiction.
\end{proof}
This metric space is always complete\cite{bourbakitop} assuming $F$ is connected simple graph multifunction. To see this notice that then $d_{F}(u,w)<\epsilon$ iff $\exists k\in N\cap(0,\epsilon)\colon u\in F^{k\cup}(w)$.
\begin{lemma}
Let $F\colon V\leadsto V$ be a connected simple graph multifunction. Then the metric space $(V,d_{F})$ is complete.
\end{lemma}

\begin{proof}
Fix $(v_{n})\in V^{N}$. Assume that $(v_{n})$ is Cauchy i.e. $\forall\epsilon>0\colon\exists n_{0}\in N\colon\forall n,m\ge n_{0}\colon\exists k\in N\cap[0,\epsilon)\colon v_{n}\in F^{k\cup}(v_{m})$. Then there exists $n_{0}\in N$ such that $\forall n,m\ge n_{0}\colon v_{n}\in F^{0\cup}(v_{m})=\{v_{m}\}$ i.e. $(v_{n})$ is eventually constant. But every eventually constant sequence from $V$ is convergent in $V$. Therefore $(V,d_{F})$ is complete. 
\end{proof}

We will check wich filters are Cauchy\cite{bourbakitop}. For this purpose we introduce leafs.
\begin{definition}
Let $F\colon V\leadsto V$ be a simple graph multifunction and $v\in V$. Then $Leaf_{v}(F,V)=\{w\in V\mid F(w)=\{v\}\}$ is a set of leafs of vertex $v$.

\end{definition}

\begin{lemma}
Let $F\colon V\leadsto V$ be a strict simple graph multifunction and $v\in V$. Then$\colon$
\begin{enumerate}[(i)]
\item$F_{+}(\{v\})=Leaf_{v}(F,V)$
\item$\{v\}\in Wall_{Leaf_{v}(F,V)}(F,V)$
\item$\{v\}\in Wall_{\aleph_{0}}(F,V)$ iff $\sharp Leaf_{v}(F,V)^{c}<\aleph_{0}$
\end{enumerate}

\end{lemma}

\begin{proof}
\begin{enumerate}[(i)]
\item$F_{+}(\{v\})=\{w\in V\mid F(w)\subset\{v\}\}=\{w\in V\mid F(w)=\{v\}\}=Leaf_{v}(F,V)$ because $\forall w\in V\colon F(w)\neq\emptyset$.
\item$\{v\}\in Wall_{Leaf_{v}(F,V)}(F,V)$ iff $Leaf_{v}(F,V)\subset F_{+}(\{v\})$ but\newline$F_{+}(\{v\})=Leaf_{v}(F,V)$.
\item$\{v\}\in Wall_{\aleph_{0}}(F,V)$ iff $\sharp Leaf_{v}(F,V)^{c}=\sharp F_{+}(\{v\})^{c}<\aleph_{0}$ i.e. the set of vertices that are not leafs for $v$ is finite.
\end{enumerate}
\end{proof}

\begin{lemma}
Let $F\colon V\leadsto V$ be a connected simple graph multifunction and $v\in V$. Then $Wall_{Leaf_{v}(F,V)}(F,V)$ is Cauchy in $(V,d_{F})$ and if $\sharp Leaf_{v}(F,V)^{c}<\aleph_{0}=\sharp(V)$ then $Wall_{\aleph_{0}}(F,V)$ is Cauchy in $(V,d_{F})$.

\end{lemma}

\begin{proof}
Fix $\epsilon>0,v\in V$. Then $diam_{F}(\{v\})=\sup_{u,w\in\{v\}}d_{F}(u,w)=0<\epsilon$ and $\{v\}\in Wall_{Leaf_{v}(F,V)}(F,V)$. Then $\forall\epsilon>0\colon\exists M\in Wall_{Leaf_{v}(F,V)}(F,V),$\newline$M=\{v\}\colon diam_{F}(M)<\epsilon$. Assume that $\sharp V=\aleph_{0}$ and $\sharp Leaf_{v}(F,V)^{c}<\aleph_{0}$. Then $\{v\}\in Wall_{\aleph_{0}}(F,V)$ and therefore $\forall\epsilon>0\colon\exists M\in Wall_{\aleph_{0}}(F,V),M=\{v\}\colon diam_{F}(M)<\epsilon$.
\end{proof}

\subsection{Multifunction of prime numbers}
We will apply the above facts about filters to divisibility of natural numbers. We will use multifunction $\text{Prime}\colon N\to P(N);n\mapsto\text{Prime}(n)=\{p\in N_{prime}\mid\exists k\in N\colon n=k\cdotp p\}$, where $N_{prime}$ is set of prime numbers. This multifunction is strict because of fundamental theorem of arithmetic\cite{vinogradov}. We introduce a bit differently this theorem. Let $F_{\aleph_{0}}=\{U\in P(N)\mid\sharp(N-U)<\aleph_{0}\}\subset P(N)$ be the cofinite filter. We denote $Ev_{0}^{prime}(N)=\{(a_{n})\in N^{N}\mid N_{prime}^{c}\subset(a_{n})^{-1}(\{0\})\in F_{\aleph_{0}}\}\subset N^{N}$ the set of $0-$eventually constant sequences that is $0$ out of $N_{prime}$. The fundamental theorem of arithmetic says that there is the bijection $\lambda\colon N\to Ev_{0}^{prime}(N);n\mapsto\lambda^{n}=(\lambda_{m}^{n})$ iff $n=\Pi_{p\in N_{prime}}p^{\lambda^{n}_{p}}$. It is easily seen that $\text{Prime}(n)=(\lambda^{n})^{-1}(\{0\})^{c}$ for each $n\in N$ and $(\text{Prime}\circ\lambda^{-1})((a_{n}))=(a_{n})^{-1}(\{0\})^{c}$ for each $(a_{n})\in Ev_{0}^{prime}(N)$.

We list some simple properties of the $\text{Prime}$ multifunction.
\begin{lemma}
Let $U\in P(N)$. Then$\colon$
\begin{enumerate}[(i)]
\item$\text{Prime}_{+}(N_{prime})=N$
\item$\reflectbox{D}(\text{Prime})=N_{prime}$
\item$\forall(a_{n})\in Ev_{0}^{prime}(N)\colon\text{Prime}(\Pi_{p\in N_{prime}\cap U}p^{a_{p}})\subset U$
\item$\text{Prime}_{+}(U)=\{n\in N\mid\exists(a_{n})\in Ev_{0}^{prime}(N)\colon n=\Pi_{p\in N_{prime}\cap U}p^{a_{p}}\}$
\item$\text{Prime}_{+}(U)^{c}=\{n\in N\mid\exists p\in N_{prime}-U\colon p\mid n\}$
\item$\text{Prime}_{-}(U)=\{n\in N\mid\exists p\in N_{prime}\cap U\colon p\mid n\}$
\item$\text{Prime}_{-}(N_{even})=N_{even}$
\item$\text{Prime}_{-}(N_{odd})=\{2\cdotp m\mid m\in N\}^{c}$ 
\end{enumerate}
\end{lemma}

\begin{proof}
Fix $(a_{n})\in Ev_{0}^{prime}(N),n\in N$.
\begin{enumerate}[(i)]
\item$\forall n\in N\colon\text{Prime}(n)\subset N_{prime}$.
\item For every $p\in N_{prime}$ there exists $n\in N,n=2\cdotp p$ such that $p\mid n$; or equivalently, $N_{prime}\subset\bigcup_{n\in N}\text{Prime}(n)=\text{Prime}_{\cup}(N)=\reflectbox{D}(\text{Prime})$.
\item Denote $(\bar{a}_{n}^{U})\colon N\to N;n\mapsto\bar{a}_{n}^{U}=\begin{cases}a_{n}$ iff $n\in U\\0$ iff $n\notin U\end{cases}$. Assume that there exists $q\in N_{prime}$ such that $q\notin U$ and $q\in\text{Prime}(\Pi_{p\in N_{prime}\cap U}p^{a_{p}})$; or equivalently, $q\mid\Pi_{p\in N_{prime}\cap U}p^{a_{p}}=\Pi_{p\in N_{prime}}p^{\bar{a}_{p}^{U}}$ and $q\notin U$. Then $q\notin U$ and there exists $k\in N$ such that $\Pi_{p\in N_{prime}}p^{\bar{a}_{p}^{U}}=q\cdotp k=\Pi_{p\in N_{prime}}p^{\lambda^{q\cdotp k}_{p}}$. But $\lambda^{q\cdotp k}_{q}\neq 0$ and $q\notin U$. Then $(\bar{a}_{n}^{U})=(\lambda^{q\cdotp k}_{n})$ but $\bar{a}_{q}^{U}=0\neq\lambda^{q\cdotp k}_{q}$, contradiction.
\item We prove the inclusion\newline$\{n\in N\mid\exists(a_{n})\in Ev_{0}^{prime}(N)\colon n=\Pi_{p\in N_{prime}\cap U}p^{a_{p}}\}\subset\text{Prime}_{+}(U)$. Assume that $\exists(a_{n})\in Ev_{0}^{prime}(N)\colon n=\Pi_{p\in N_{prime}\cap U}p^{a_{p}}$. Then\newline$\text{Prime}(n)\subset U$ i.e. $n\in\text{Prime}_{+}(U)$. In the opposite direction we prove the inclusion $\text{Prime}_{+}(U)\subset\{n\in N\mid\exists(a_{n})\in Ev_{0}^{prime}(N)\colon n=\Pi_{p\in N_{prime}\cap U}p^{a_{p}}\}$. Then there exists only one $(\lambda^{n}_{m})\in Ev_{0}^{prime}(N)$ such that $n=\Pi_{p\in N_{prime}}p^{\lambda^{n}_{p}}$. Assume that $n\in\text{Prime}_{+}(U)$ and $\forall(a_{n})\in Ev_{0}^{prime}(N)\colon n\neq\Pi_{p\in N_{prime}\cap U}p^{a_{p}}$; or equivalently, $\forall q\in N_{prime}\colon q\mid n\Rightarrow q\in U$ and $\forall(a_{n})\in Ev_{0}^{prime}(N)\colon n\neq\Pi_{p\in N_{prime}\cap U}p^{a_{p}}$. Denote $(\bar{a}_{n}^{U})\colon N\to N;n\mapsto\bar{a}_{n}^{U}=\begin{cases}a_{n}$ iff $n\in U\\0$ iff $n\notin U\end{cases}$. We may say equivalently that $\forall q\in N_{prime}\colon q\mid n\Rightarrow q\in U$ and\newline$\forall(a_{n})\in Ev_{0}^{prime}(N)\colon\Pi_{p\in N_{prime}}p^{\lambda^{n}_{p}}\neq\Pi_{p\in N_{prime}}p^{\bar{a}_{p}^{U}}$. But $(\lambda^{n}_{m})\in Ev_{0}^{prime}(N)$ and so $\Pi_{p\in N_{prime}}p^{\lambda^{n}_{p}}\neq\Pi_{p\in N_{prime}}p^{\bar{\lambda}_{p}^{n,U}}$. Then there exists $q\in N_{prime}$ such that $\lambda^{n}_{q}\neq 0=\bar{\lambda}^{n,U}_{q}$ i.e. $q\notin U$ and $q\mid n$. But $\forall q\in N_{prime}\colon q\mid n\Rightarrow q\in U$, contradiction.
\item There exists $p\in N_{prime}$ such that $p\mid n$ and $p\notin U$ iff $\text{Prime}(n)\cap U^{c}\neq\emptyset$; or equivalently, not $\text{Prime}(n)\subset U$ i.e. $n\in\text{Prime}_{+}(U)^{c}$.
\item There exists $p\in N_{prime}$ such that $p\mid n$ and $p\in U$ iff $\text{Prime}(n)\cap U\neq\emptyset$ i.e. $n\in\text{Prime}_{-}(U)$.
\item$\text{Prime}_{-}(N_{even})=\{n\in N\mid\exists p\in N_{prime}\cap N_{even}\colon p\mid n\}=\{n\in N\mid\exists k\in N\colon n=2\cdotp k\}=N_{even}$ because $N_{prime}\cap N_{even}=\{2\}$.
\item$\text{Prime}_{+}(N_{even})=\{n\in N\mid\exists(a_{n})\in Ev_{0}^{prime}(N)\colon$\newline$n=\Pi_{p\in N_{prime}\cap N_{even}=\{2\}}p^{a_{p}}\}=\{n\in N\mid\exists m\in N\colon n=2^{m}\}=\{2^{m}\mid m\in N\}$ and $\text{Prime}_{-}(N_{odd})=\text{Prime}_{-}(N_{even}^{c})=\text{Prime}_{+}(N_{even})^{c}=\{2^{m}\mid m\in N\}^{c}$. 

\end{enumerate}
\end{proof}

We look for the set of natural numbers such that the amount of natural numbers that have a prime divisor out of one set is finite. 
\begin{lemma}$Wall_{\aleph_{0}}(\text{Prime},N)=\{U\subset N\mid$ the set of natural numbers that have a prime divisor out of $U$ is finite$\}$ is a proper filter but not ultrafilter.
\end{lemma}
\begin{proof}
Fix $U\subset N$. Then $U\in Wall_{\aleph_{0}}(\text{Prime},N)$ iff $\sharp\text{Prime}_{+}(U)^{c}<\aleph_{0}$; or equivalently, $\text{Prime}_{+}(U)^{c}$ is finite. But $\text{Prime}_{+}(U)^{c}$ is the set of natural numbers that have prime divisor out of $U$. The filter $Wall_{\aleph_{0}}(\text{Prime},N)$ is proper because $\text{Prime}$ is strict. Notice that there exists $U\subset N,U=N_{even}$ such that $\sharp\text{Prime}_{+}(U^{c})^{c}=\sharp\text{Prime}_{-}(U)=\sharp U=\aleph_{0}$ and $\sharp\text{Prime}_{+}(U)^{c}=\sharp\text{Prime}_{-}(U^{c})=\sharp\{2^{m}\mid m\in N\}^{c}=\aleph_{0}$ i.e. there exists $U\in P(N)$ such that $U\notin Wall_{\aleph_{0}}(\text{Prime},N)$ and $U^{c}\notin Wall_{\aleph_{0}}(\text{Prime},N)$. Therefore\newline$Wall_{\aleph_{0}}(\text{Prime},N)$ is not an ultrafilter. 
\end{proof}

We denote $Prime=(\text{Prime}\cup\text{Prime}^{-1})-\{\cdotp\}$ and $\bar{N}=N-\{0,1\}$. 
\begin{lemma}
For every $p\in N_{prime}$ the equality
$Leaf_{p}(Prime,\bar{N})=\{p^{m}\mid m\in N\}$ is fulfilled and $Wall_{\{p^{m}\mid m\in N\}}(Prime,\bar{N})$ is Cauchy in $(\bar{N},d_{Prime})$. 
\end{lemma}

\begin{proof}
Notice that $Leaf_{p}(Prime,\bar{N})=\{n\in N\mid Prime(n)=\{p\}\}=\{n\in N\mid\exists k\in N\colon n=k\cdotp p\land\forall(q,l)\in(N_{prime}-\{p\})\times N\colon n\neq l\cdotp q\}=\{n\in N\mid\exists m\in N\colon n=p^{m}\}=\{p^{m}\mid m\in N\}$. Therefore $Wall_{\{p^{m}\mid m\in N\}}(Prime,\bar{N})$ is Cauchy in $(\bar{N},d_{Prime})$. 
\end{proof}

\section*{Conclusions}
We have introduced the multifunction iteration theory in detail. We translated some ideas of graph theory into multifunction language. We showed that there are graph theoretic filters and ideals. We believe that thanks to this, other readers will share the belief that everything from graph theory can be translated into multifunctions and that, in most cases, this translation will enrich the analysis process. In the future we want to answer the following questions$\colon$
\begin{enumerate}
\item how to describe the multipartite graphs using iterations
\item how to detect the even cycle in graph using iterations
\item what facts from discrete dynamical systems on multifunctions are useful for graph theory
\end{enumerate}


\end{document}